\documentclass[10pt]{amsart}

\usepackage{mathtools}
\usepackage[round]{natbib}
\setcitestyle{numbers,square}
\usepackage{hyperref}
\usepackage{xcolor}
\usepackage[normalem]{ulem}
\usepackage{tikz}

\usepackage{accents}
\newcommand\redsout{\bgroup\markoverwith{\textcolor{red}{\rule[0.5ex]{2pt}{0.4pt}}}\ULon}

\newcommand{\E}{\mathbb{E}}
\newcommand{\N}{\mathbb{N}}

\newcommand{\R}{\mathbb{R}}

\newcommand{\Pb}{\mathbb{P}}
\newcommand{\tp}{t^{\prime}}
\newcommand{\pa}{a^{\prime}}
\newcommand{\kp}{k^{\prime}}

\newcommand{\ps}{s^{\prime}}

\newcommand{\ve}{\varepsilon}

\newcommand{\ovu}{\overline{u}}

\newcommand{\und}{\underline{u}}

\newcommand{\wdu}{\widehat{u}}

\newcommand{\ze}{\zeta}

\newcommand{\vep}{\varepsilon^{\prime}}

\newcommand{\diffns}{\mathrm{d}}

\def\={{\;\mathop{=}\limits^{\text{(law)}}\;}}

\newtheorem{theorem}{Theorem}[section]
\newtheorem{prop}[theorem]{Proposition}
\newtheorem{lemma}[theorem]{Lemma}
\newtheorem{defi}[theorem]{Definition}
\newtheorem{corollary}[theorem]{Corollary}
\newtheorem{hyp}[theorem]{Hypothesis} 

\newtheorem{remark}[theorem]{Remark}

\numberwithin{equation}{section}

\usepackage[english]{babel}

\title[Malliavin differentiability of solutions of hyperbolic SPDEs with irregular drifts]%\texttt{Regularization by noise for SDEs in the plane}]
{Malliavin differentiability of solutions of hyperbolic stochastic partial differential equations with irregular drifts }%{Regularisation by noise for Stochastic Differential Equations in the plane with Nondecreasing drifts}
\author[A.-M. Bogso]{Antoine-Marie Bogso}
\address{University of Yaounde I\\
	Faculty of Sciences, Department of Mathematics\\
	P.O. Box 812, Yaounde, Cameroon \\
	 and African institute for Mathematical Sciences Ghana, P.O. Box LGDTD 20046, Summerhill Estates, East Legon Hills, Santoe, Acrra}
\email{antoine.bogso@facsciences-uy1.cm, antoine@aims.edu.gh}           %  \\

\author[O. Menoukeu Pamen]{Olivier Menoukeu Pamen}
\address{Institute for Financial and Actuarial Mathematics (IFAM) \\
	Department of Mathematical Sciences, University of Liverpool \\
	Liverpool L69 7ZL, UK \\
	and AIMS Ghana}
\email{menoukeu@liverpool.ac.uk}
\thanks{The project on which this publication is based has been carried out with funding provided by the Alexander von Humboldt Foundation, under the programme financed by the German Federal Ministry of Education and Research entitled German Research Chair No 01DG15010.}


\date{\today}

\begin{document}

 \begin{abstract}
 	We prove path-by-path uniqueness of solution to hyperbolic stochastic partial differential equations when the drift coefficient is the difference of two componentwise monotone Borel measurable functions of spatial linear growth. The Yamada-Watanabe principle for SDE driven by Brownian sheet then allows to derive strong uniqueness for such equation and thus extending the results in [Bogso, Dieye and Menoukeu Pamen, Elect. J. Probab.,  27:1-26, 2022] and [Nualart and Tindel, Potential Anal., 7(3):661--680, 1997]. Assuming that the drift is globally bounded, we show that the unique strong solution is Malliavin differentiable. The case of spatial linear growth drift coefficient is also studied.
 \end{abstract}

\maketitle 

\section{Introduction}
\label{intro}

The existence, uniqueness and Malliavin differentiability of strong solutions of SDEs on the plane with smooth coefficients have been obtained in several settings of varying generality. However there are not many results when the coefficient of the such equation are singular. The purpose of the present paper is two-fold: first we obtain the existence and uniqueness of strong solution of the following integral form equation
\begin{align}\label{Eqmainhpde4}
	X_{s,t}=\xi+\int_0^t\int_0^sb(s_1,t_1,X_{s_1,t_1})\mathrm{d}s_1\mathrm{d}t_1+W_{s,t}, 
	\text{ for }(s,t)\in\Gamma,
\end{align}
when the drift $b$ is the difference of two componentwise monotone functions and of spatial linear growth. We address this problem by using the Yamada-Watanabe argument for SDEs driven by Brownian sheet  derived in \cite{NuYe89} (see also \cite{Ye87}, \cite[Remark 2]{Tu83}), that is, we combine weak existence and pathwise uniqueness to obtain the existence of a unique strong solution. More particularly, we replace the pathwise uniqueness by a stronger notion of uniqueness, namely, the path-by-path uniqueness introduced in \cite{Da07} (see also \cite{Fla10}) in the ecase of SDEs driven by one-parameter Brownian motion. This notion was introduced in \cite{BDM21b} for the two parameter process, as follows:

\begin{defi}\label{defipathbpath}
	We say that the path-by-path uniqueness of solutions to \eqref{Eqmainhpde4} holds when there exists a full $\Pb$-measure set $\Omega_0\subset\Omega$ such that for all $\omega\in\Omega_0$ the following statement is true: there exists at most one function  $y\in\mathcal{V}$ which satisfies $$\int_0^T\int_0^T|b(s,t,y_{s,t})|\mathrm{d}s \mathrm{d}t<\infty,\text{ }\partial y=x,\text{ for some }x\in\partial \mathcal{V}\text{ and }T>0$$ and
	\begin{align}\label{eqmainpathbypath}
		y_{s,t}=x+\int_0^s\int_0^tb(s_1,t_1,y_{s_1,t_1})\mathrm{d}s_1 \mathrm{d}t_1+W_{s,t}(\omega),\text{ }\forall\,(s,t)\in[0,T]^2.
	\end{align}
\end{defi}

The study of path-by-path uniqueness is motivated by the problem of regularisation by noise of random ordinary (or partial) differential equations (ODEs or PDEs).  In the case of of SDEs driven by Brownian motion, path-by-path uniqueness of equation \eqref{Eqmainhpde4} was proved in \cite{Da07} assuming that the drift is bounded and measurable, and the diffusion is constant. This result was extended to the non-constant diffusion in \cite{Da11} using rough path analysis. There has now been several generalisation of this result.  The authors in \cite{BFGM14} proved a Sobolev regularity of solutions to the linear stochastic transport and continuity equations with drift in critical $L^p$ spaces. Such a result does not hold for the corresponding deterministic equations.  In \cite{BM16}, the authors analysed the regularisation by noise  for a non-Lipschitz stochastic heat equation and proved path-by-path uniqueness for any initial condition in a certain class of a set of probability one. In \cite{AMP20}, the path-by-path uniqueness for transport equations driven by \textcolor{black}{the fractional Brownian motion of Hurst index $H<1/2$ with bounded and integrable vector-fields} is investigated. In \cite{CG16,GG21} the authors solved the regularisation by noise problem from the point of view of additive perturbations. In particular, the work \cite{CG16} considered generic perturbations without any specific probabilistic setting whereas authors in \cite{ABP17a} construct a new Gaussian noise of fractional nature and proved that it has a strong regularising effect on a large class of ODEs. More recently, the regularisation by noise problem for ODEs with vector fields given by Schwartz distributions was studied in \cite{HP21}. It was also proved that  if one perturbs such an equation by adding an infinitely regularising path, then it has a unique solution. Let us also mention the recent work \cite{KP20a} in which the authors looked at multidimensional SDEs with distributional drift driven by symmetric $\alpha$-stable L\'evy processes for $\alpha\in(1,2]$. In all of the above mentioned works, the driving noise considered are one parameter processes.

Our method to prove path-by-path uniqueness follows as in \cite{BDM21b}. We show the path-by-path uniqueness on $\Gamma_0=[0,1]^2$. More precisely, we consider the integral equation  
\begin{align}\label{Eqmainhpde5}
	X_{s,t}=\xi+\int_0^t\int_0^sb(s_1,t_1,X_{s_1,t_1})\mathrm{d}s_1\mathrm{d}t_1+W_{s,t}
	\text{ for }(s,t)\in \Gamma_0,
\end{align}
where the drift is of spatial linear growth. It is shown in \cite[Section 1]{BDM21b} (see also \cite[Section 1]{Da07}) that  path-by-path uniqueness of solutions to \eqref{Eqmainhpde5} holds if and only if, with probability one, there is no nontrivial solution $u\in \mathcal{V}^1_0$ of 
\begin{align}\label{Eqmainhpde6}
	u(s,t)=\int_0^s\int_0^t\{b(s_1,t_1,W_{
		s_1,t_1}+u(s_1,t_1))-b(s_1,t_1,W_{s_1,t_1})\}\mathrm{d}s_1 \mathrm{d}t_1,\text{ for }(s,t)\in \Gamma_0.
\end{align} 	
This is the statement of Theorem \ref{maintheuniq2} which is extended to unbounded monotone drifts in Theorem \ref{maintheuniq1}. 
The proof of Theorem \ref{maintheuniq2} relies on some estimates for an averaging  operator along the sheet (see Lemma \ref{lem:PseudoMetric1}). This result plays a key role in the proof of a Gronwall type lemma (see Lemma \ref{lem:GronwallSheet}) which enables us  to prove path-by-path uniqueness of solutions to \eqref{Eqmainhpde5}. The latter combined with the weak existence yield the existence of a unique strong solution. Note that when the drift coefficient is , similar result can be found in \cite{BDM21b, NuTi97}.

Secondly, in this paper, we prove Malliavin smoothness of the unique solution to the SDE \eqref{Eqmainhpde5}. When the coefficient are smooth, the authors in \cite{NuSa85, NuSa89} showed existence, uniqueness, Malliavin differentiability and smoothness of density of solutions to SDEs on the plane. Here, assuming that the drift is the difference of two componentwise nondecreasing functions, we show that the solution is Malliavin differentiable. In the one parameter case, the Malliavin differentiablity of  solutions to SDEs with bounded and measeurable coefficients was studied in \cite{MBP10} under an additional commutativity assumption. The later assumption was removed in \cite{MMNPZ13}.  It is worth mentioning that in the above work, the Malliavin smoothness of the unique solutions to the SDEs with rough coefficients and driven by Brownian motion is obtained as a byproduct of the method used to study existence and uniqueness. This technique was introduced in \cite{Pro07} and has now been extensively utilised; see for example the work \cite{HP14} for the case of singular SDEs driven by L\'evy noise, \cite{MenTan19} for the case of random coefficients and \cite{AMP20, ABP17a} for the case of  singular SDEs driven by fractional noise.  In order to prove the Malliavin differentiability of the solution to the SDE \eqref{Eqmainhpde5}, we take advantage of Gaussian white noise theory and a local time-space integration formula provided in \cite[Proposition 3.1]{BDM21a} to show that the sequence of approximating sequence of solution converges strongly in $L^2(\Omega,\mathbb{R}^d)$ to the solution of the SDE (compare with \cite{MMNPZ13}) and we use a compactness criteria given in \cite[Lemma 1.2.3]{Nu06} to conclude.
\newpage

Equation \eqref{Eqmainhpde4} can also be written in a differential form as the following hyperbolic stochastic partial differential equation 
\begin{align}\label{Eqmainhpde1}
	\left\{
	\begin{array}{ll}
		\dfrac{\partial^2X_{s,t}}{\partial s\partial t}=b(s,t,X_{s,t})+\dot{W},&(s,t)\in\Gamma,\\&\\
		\partial X=\xi,
	\end{array}
	\right.
\end{align}  
where $\partial X$ is the restriction of $X$ to the boundary $\partial \Gamma=\{0\}\times\R_+\cup\R_+\times\{0\}$ of $\Gamma:=\R_+^2$, $b:\Gamma\times \mathbb{R}^d\to\R^d$ is Borel measurable, $\dot{W}=(\dot{W}^{(1)},\ldots,\dot{W}^{(d)})$ is a $d$-dimensional  white noise on $\Gamma$ given on a probability space $(\Omega,\mathcal{F},\Pb)$ and $(s,t)\longmapsto\xi_{s,t}(\omega)$ is continuous on $\partial \Gamma$ for all $\omega\in\Omega$. Recall that by a $d$-dimensional white noise on $\Gamma$ we mean a mean-zero Gaussian process $\dot{W}=(\dot{W}^{(1)},\ldots,\dot{W}^{(d)})$ indexed by the Borel field $\mathcal{B}(\Gamma)$ on $\Gamma$ with covariance functions
$$
\E\Big[\dot{W}^{(i)}(A)\dot{W}^{(j)}(B)\Big]=\delta_{i,j}|A\cap B|,\quad\forall\,A,B\in\mathcal{B}(\Gamma)
$$ 
where $|\cdot|$ denotes the Lebesgue measure on $\Gamma$ and  $\delta_{i,j}=1$ if $i=j$ and $\delta_{i,j}=0$ otherwise. The process $W=\left(W_{s,t}:=\dot{W}([0,s]\times[0,t]),(s,t)\in \Gamma\right)$ is mean-zero Gaussian process with covariance functions
$$
\E\Big[W^{(i)}_{s,t}W^{(j)}_{\ps,\tp}\Big]=\delta_{i,j}(s\wedge \ps)(t\wedge\tp),\quad\forall\,(s,t),(\ps,\tp)\in \Gamma.
$$
By the Kolmogorov continuity theorem, there exists a continuous version of $W$, still denoted by $W$, which is a  $d$-dimensional Brownian sheet. We consider a 
nondecreasing and right-continuous family $\mathbb{F}=(\mathcal{F}_{s,t})$ of sub-$\sigma$-algebras of $\mathcal{F}$ each of which contains all negligible sets in $(\Omega,\mathcal{F},\Pb)$ such that $W$ and $\xi$ are $\mathbb{F}$-adapted, that is $W_{s,t}$ (respectively, $\xi_{s,t}$) is $\mathcal{F}_{s,t}$-measurable for every $(s,t)\in \Gamma$ (respectively, $(s,t)\in\partial \Gamma$).
We refer the reader to Khoshnevisan \cite{Kh02} for a complete
analysis on multi-parameter processes and their applications.

Equation \eqref{Eqmainhpde1} is a particular case of the quasilinear  stochastic hyperbolic differential equation 
\begin{align}\label{Eqmainhpde2}
	\left\{
	\begin{array}{ll}
		\dfrac{\partial^2 X_{s,t}}{\partial s\partial  t}=b(s,t,X_{s,t})+a(s,t,X_{s,t})\dot{W},&\quad(s,t)\in \Gamma\\&\\
		\partial X=\xi,&
	\end{array}
	\right.  	
\end{align}
where  $a:\,\R_+^2\times\R^d\to\R^d\times\R^d$ is a Borel measurable matrix function. A formal $\frac{\pi}{4}$ rotation transforms \eqref{Eqmainhpde2} into the following nonlinear stochastic wave equation 
\begin{align}\label{Eqmainhpde3}
	\dfrac{\partial^2 Y_{\rho,\theta}}{\partial \rho^2}-\dfrac{\partial^2 Y_{\rho,\theta}}{\partial \theta^2}=\tilde{b}(\rho,\theta,Y_{\rho,\theta})+\tilde{a}(\rho,\theta,Y_{\rho,\theta})\dot{\tilde{W}},\quad(\rho,\theta)\in \tilde{\Gamma},
\end{align}
with the Goursat-Darboux type boundary condition $\partial Y=\tilde{\xi}$,
where $\tilde{\Gamma}=\{(\rho,\theta):\,\theta\geq0\text{ and }|\rho|\leq \theta\}$, $\dot{\tilde{W}}$ is a $d$-dimensional white noise on $\tilde{\Gamma}$, $\tilde{b}(\rho,\theta,y)=b(\frac{\theta+\rho}{\sqrt{2}},\frac{\theta-\rho}{\sqrt{2}},y)$ (the same applies to $\tilde{a}$), $Y_{\rho,\theta}= X_{\frac{\theta+\rho}{\sqrt{2}},\frac{\theta-\rho}{\sqrt{2}}}$,  $\tilde{\xi}_{\theta,\theta}=\xi_{\sqrt{2}\theta,0}$ and $\tilde{\xi}_{-\theta,\theta}=\xi_{0,\sqrt{2}\theta}$. The $\frac{\pi}{4}$ rotation has been used by Carmona and Nualart \cite{CaNu88} (see also \cite[Section 0]{FaNu93} and \cite[Section 1]{QuTi07}) to prove existence and uniqueness of solution to \eqref{Eqmainhpde3} under a different boundary condition when $\tilde{a}$ and $\tilde{b}$ are time-homogeneous. 

\eqref{Eqmainhpde1} can also be seen as a noisy analog of the so-called Darboux problem given by
\begin{align}\label{eqdarboux1}
	\frac{\partial^2y}{\partial s\partial t}=b\Big(s,t,y,\frac{\partial y}{\partial s},\frac{\partial y}{\partial t}\Big) \quad\text{for }(s,t)\in[0,T]\times[0,T],
\end{align}
with the initial conditions 
\begin{align}\label{eqdarboux2}
	y(0,t)=\sigma(t)\,\text{ on }[0,T]\,\text{ and }\,y(s,0)=\tau(s)\,\text{ on }[0,T],
\end{align}
where $\sigma$ and $\tau$ are absolutely continuous on $[0,T]$.
\textcolor{black}{Using Caratheodory's theory} of differential equations, Deimling \cite{De70} proved an existence theorem for the system \eqref{eqdarboux1}-\eqref{eqdarboux2} \textcolor{black}{when} $b$ is Borel measurable in the first two variables and bounded and continuous in the last three variables.  Hence the results obtained here can also be seen as a generalisation to the stochastic setting of the above mentioned one. 

The remainder of the paper is organised as follows: In Section \ref{sec2}, we provide a path-by-path uniqueness result for \eqref{Eqmainhpde5} when the drift $b$ is of linear growth. %In particular, using the Yamada-Watanabe principle, we deduce the existence of a unique strong solution. 
In Section \ref{Malreg}, we study the Malliavin differentiability of the strong solution to \eqref{Eqmainhpde5}. We show that this  solution  is Malliavin differentiable for uniformly bounded drifts and when the drift $b$ is of linear growth, we obtain Malliavin differentiability of the solution only for sufficiently small time parameters.

\section{Existence and uniqueness results}\label{sec2}
In this section, we show that the SDE \eqref{Eqmainhpde5} has a unique strong solution. Our approach is based on the Yamada-Watanabe principle introduced in \cite{NuYe89}. As pointed out earlier, instead of showing the  weak existence and pathwise uniqueness, we show weak existence and path by path uniqueness (which implies pathwise uniqueness as shown in \cite{CG16}).  The following preliminary results that have been obtained by applying a local time-space integration formula for Brownian sheets (see \cite{BDM21a} for more information) are needed to show  path-by-path uniqueness.

\subsection{Preliminary results}\label{prelresul}

Let $f:\,[0,1]^2\times\R^d\to\R$ be a continuous function such that for any $(s,t)\in[0,1]^2$, $f(s,t,\cdot)$ is differentiable and for any $i\in\{1,\cdots,d\}$, the partial derivative $\partial_{x_i}f$ is continuous. We also know from \cite[Proposition 3.1]{BDM21a} that for a $d$-dimensional Brownian sheet $\Big(W_{s,t}:=(W_{s,t}^{(1)},\cdots,W_{s,t}^{(d)});s\geq0,t\geq0\Big)$ defined on a filtered probability space and for any $(s,t)\in[0,1]^2$ and any $i\in\{1,\cdots,d\}$, we have
\begin{align}\label{eq:EisenSheetdD01}
	&\int_0^s\int_0^t\partial_{x_i}f(s_1,t_1,W_{s_1,t_1})\mathrm{d}t_1\mathrm{d}s_1\notag\\
	=&-\int_0^s\int_0^tf(s_1,t_1,W_{s_1,t_1})\frac{d_{t_1}W^{(i)}_{s_1,t_1}}{s_1}\mathrm{d}s_1-\int_0^s\int_{1-t}^1f(s_1,1-t_1,\widehat{W}_{s_1,t_1})\frac{d_{t_1}B^{(i)}_{s_1,t_1}}{s_1}\mathrm{d}s_1\\
	&+\int_0^s\int_{1-t}^1f(s_1,1-t_1,\widehat{W}_{s_1,t_1})\frac{\widehat{W}^{(i)}_{s_1,t_1}}{s_1(1-t_1)}\mathrm{d}t_1\mathrm{d}s_1,\nonumber
\end{align}
where $\widehat{W}^{(i)}:=(\widehat{W}^{(i)}_{s_1,t_1};0\leq s_1,t_1\leq1)$ and $B^{(i)}:=(B^{(i)}_{s_1,t_1};0\leq s_1,t_1\leq1)$ is a standard Brownian sheet with respect to the filtration of $\widehat{W}^{(i)}$, independent of $(W^{(i)}_{s,1},s\geq0)$.

The following result will be extensively used and can be found in \cite{BDM21b}.

\begin{prop}\label{prop:DavieSheet1dd}
	Let $W:=\left(W^{(1)}_{s,t},\ldots,W^{(d)}_{s,t};(s,t)\in[0,1]^2\right)$ be a $\R^d$-valued Brownian sheet defined on a filtered probability space $(\Omega,\mathcal{F},\mathbb{F},\Pb)$, where $\mathbb{F}=(\mathcal{F}_{s,t};s,t\in[0,1])$. Let $b\in\mathcal{C}\left([0,1]^2,\mathcal{C}^1(\R^d)\right)$, $\Vert b\Vert_{\infty}\leq1$.
	Let $(a,\pa,\ve,\vep)\in[0,1]^4$. Then there exist positive constants $\alpha$ and $C$ (independent of $\nabla_yb$, $a$, $\pa$, $\ve$ and $\vep$) such that 
	\begin{align}\label{eq:DavieSheet0dd}
		\E\Big[\exp\Big(\alpha\vep \ve \Big|\int_0^1\int_0^1\nabla_yb\left(s,t,\widetilde{W}^{\ve,\vep}_{s,t}\right)\mathrm{d}t\mathrm{d}s\Big|^2\Big)\Big]\leq C.
	\end{align} 
	Here $\nabla_yb$ denotes the gradient of $b$ with respect to the third variable, $|\cdot|$ is the usual norm on $\R^d$ and  the $\R^d$-valued two-parameter Gaussian process $\widetilde{W}^{\ve,\vep}:=\Big(\widetilde{W}^{(\ve,\vep,1)}_{s,t},\ldots,\widetilde{W}^{(\ve,\vep,d)}_{s,t};(s,t)\in[0,1]^2\Big)$ is given by 
	$$
	\widetilde{W}^{(\ve,\vep,i)}_{s,t}=W^{(i)}_{\pa+\vep s,a+\ve t}-W^{(i)}_{\pa,a+\ve t}-W^{(i)}_{\pa+\vep s,a}+W^{(i)}_{\pa,a}\quad\text{for all }i\in\{1,\ldots,d\}.
	$$ 
\end{prop}

\textcolor{black}{For every $0\leq a< \gamma\leq 1$, $0\leq \pa< \gamma'\leq 1$ and for $(x,y)\in\mathbb{R}^d$  let us define the function $\varrho$ by: $$\varrho(x,y)=\int_{\pa}^{\gamma'}\int_a^{\gamma}\Big\{b(s,t,W_{s,t}+x)-b(s,t,W_{s,t}+y)\Big\}\mathrm{d}t\mathrm{d}s.$$
	Here is a direct consequence of the previous estimation.}
\begin{corollary}\label{corol:DavieSheet1dds1}
	Let $b:[0,1]^2\times\R^d\to\R$ be a bounded and Borel measurable function such that $\Vert b\Vert_{\infty}\leq1$. Let $\alpha$, $C$ and $\widetilde{W}^{\ve,\vep}$ be defined as in Proposition  \ref{prop:DavieSheet1dd}. Then the following two bounds are valid: 
	\begin{enumerate}
		\item For every $(x,y)\in\R^{2d}$, $x\neq y$ and every $(\ve,\vep)\in[0,1]^2$, we have
		\begin{align}\label{eq:DavieSheet02dd}
			\E\Big[\exp\Big(\frac{\alpha\vep \ve }{|x-y|^2}\Big|\int_0^1\int_0^1\left\{b(s,t,\widetilde{W}^{\vep,\ve}_{s,t}+x)-b(s,t,\widetilde{W}^{\vep,\ve}_{s,t}+y)\right\}\mathrm{d}t\mathrm{d}s\Big|^2\Big)\Big]\leq C.
		\end{align}
		\item  
		For any $(x,y)\in\R^2$ and any $\eta>0$, we have
		\textcolor{black}{\begin{align}\label{eq:EstDavieSigma1}
				\Pb\left(|\varrho(x,y)|\geq\eta\sqrt{(\gamma-a)(\gamma'-\pa)}|x-y|\right)
				\leq Ce^{-\alpha\eta^2}.
			\end{align}
		}
	\end{enumerate}
\end{corollary}

For any positive integer $n$, we divide $[0,1]$ into $2^{n}$ intervals $I_{nk}=]k2^{-n},(k+1)2^{-n}]$. We define the random real valued function $\varrho_{nk\kp}$ on $[-1,1]^{2d}$ by
\begin{align*}
	\varrho_{nk\kp}(x,y):=\int_{I_{n\kp}}\int_{I_{nk}}\{b(s,t,W_{s,t}+x)-b(s,t,W_{s,t}+y)\}\,\mathrm{d}t\mathrm{d}s.
\end{align*}
The next two lemmas provide an estimate for $\varrho_{nk\kp}(x,y)$ and $\varrho_{nk\kp}(0,x)$ for every dyadic numbers $x,y\in\mathbf{Q}:=[-1,1]^d$. Their proofs can be found in \cite[Section 5]{BDM21b}. %The second lemma plays a central role in the proof of the path by path uniqueness. 
\begin{lemma}\label{lem:PseudoMetric1}
	Suppose $b:\,[0,1]^2\times\R^d\to\R$ is a Borel measurable function such that \textcolor{black}{$|b(s,t,x)|\leq1$} everywhere on $[0,1]^2\times\R^d$. Then there exists a subset $\Omega_{1}$ of $\Omega$ with $\Pb(\Omega_{1})=1$ such that for all $\omega\in\Omega_1$,
	\begin{align*}
		|\varrho_{nk\kp}(x,y)(\omega)|\leq C_1(\omega)2^{-n}\Big[\sqrt{n}+\Big(\log^+\frac{1}{|x-y|}\Big)^{1/2}\Big]|x-y|\,\text{ on }\Omega_{1}
	\end{align*}
	for all dyadic numbers $x,\,y\in\mathbf{Q}$ and all choices of integers $n,\,k,\,\kp$ with $n\geq1$, $0\leq k,\kp\leq 2^n-1$, where $C_1(\omega)$ is a positive random constant that does not depend on $x$, $y$, $n$, $k$ and $\kp$.
\end{lemma}
\begin{lemma}\label{lem:PseudoMetric2}
	Suppose $b$ is as in Lemma \ref{lem:PseudoMetric1}. Then there exists a subset $\Omega_{2}$ of $\Omega$ with $\Pb(\Omega_{2})=1$  such that for all $\omega\in\Omega_2$, for any choice of $n,\,k,\,\kp$, and any choice of a dyadic number $x\in\mathbf{Q}$
	\begin{align}\label{eq:coPseudoMetric2}
		\left|\varrho_{nk\kp}(0,x)(\omega)\right|\leq C_{2}(\omega)\sqrt{n}2^{-n}\Big(|x|+2^{-4^{n}}\Big),
	\end{align}
	where $C_2(\omega)$ is a positive random constant that does not depend on $x$, $n$, $k$ and $\kp$.
\end{lemma}
Observe that  
the above two results require only required the drift  $b$ to be bounded and Borel measurable. 
Assuming in addition $b$ is nondecreasing, the next two results state that \textcolor{black}{Lemmas \ref{lem:PseudoMetric1} and} \ref{lem:PseudoMetric2} can be extended to any \textcolor{black}{$x,y\in\mathbf{Q}$} (not only dyadic). The proof of Lemma \ref{lem:PseudoMetric3} is omitted since it is similar to that of Lemma \ref{lem:PseudoMetric1a}.
\textcolor{black}{
	\begin{lemma}\label{lem:PseudoMetric1a}
		Suppose $b$, $\Omega_1$ and $C_1$ are as in Lemma \ref{lem:PseudoMetric1}. Suppose in addition that $b$ is   componentwise nondecreasing. Then for all $\omega\in\Omega_1$,
		\begin{align*}
			|\varrho_{nk\kp}(x,y)(\omega)|\leq C_1(\omega)2^{-n}\Big[\sqrt{n}+\Big(\log^+\frac{1}{|x-y|}\Big)^{1/2}\Big]|x-y|\,\text{ on }\Omega_{1}
		\end{align*}
		for all $x,\,y\in\mathbf{Q}$ and all choices of integers $n,\,k,\,\kp$ with $n\geq1$, $0\leq k,\kp\leq 2^n-1$.
	\end{lemma}
	\begin{proof}
		Fix $\omega\in\Omega_1$, $x,y\in\mathbf{Q}$, $n\geq1$ and $0\leq k,\kp\leq 2^n-1$. Suppose without loss of generality that $\varrho_{nk\kp}(x,y)(\omega)>0$. For every $i\in\{1,\ldots,d\}$ and $\ell\in\N$, define   $y^-_{i,\ell}=2^{-\ell}[2^{\ell}y_i]$, $x^+_{i,\ell}=1-2^{-\ell}[2^{\ell}(1-x_i)]$,   $y^-_{\ell}=(y^-_{1,\ell},\ldots,y^-_{d,\ell})$ and $x^+_{\ell}=(x^+_{1,\ell},\ldots,x^+_{d,\ell})$. Observe that $x^+_{\ell}$ (respectively $y^-_{\ell}$) is a componentwise non-increasing (respectively non-decreasing) sequence of dyadic vectors that converges to $x$ (respectively $y$). Hence, as $b(s,t,W_{s,t}(\omega)+x)\leq b(s,t,W_{s,t}(\omega)+x^+_{\ell})$ and $b(s,t,W_{s,t}(\omega)+y^-_{\ell})\leq b(s,t,W_{s,t}(\omega)+y)$, it follows from Lemma \ref{lem:PseudoMetric1} that
		\begin{align*}
			|\varrho_{nk\kp}(x,y)(\omega)|&=\varrho_{nk\kp}(x,y)(\omega)=\int_{I_{n\kp}}\int_{I_{nk}}\{b(s,t,W_{s,t}(\omega)+x) -b(s,t,W_{s,t}(\omega)+y)\}\,\mathrm{d}t\mathrm{d}s\\ \leq&\int_{I_{n\kp}}\int_{I_{nk}}\{b(s,t,W_{s,t}(\omega)+x^+_{\ell}) -b(s,t,W_{s,t}(\omega)+y^-_{\ell})\}\,\mathrm{d}t\mathrm{d}s\\ \leq&C_1(\omega)2^{-n}\Big[\sqrt{n}+\Big(\log^+\frac{1}{|x^+_{\ell}-y^-_{\ell}|}\Big)^{1/2}\Big]|x^+_{\ell}-y^-_{\ell}|.
		\end{align*}
		Then, letting $\ell$ tends to $\infty$, we obtain the result.
\end{proof}}
\textcolor{black}{
	\begin{lemma}\label{lem:PseudoMetric3}
		Suppose $b$, $\Omega_2$ and $C_2$ are as in Lemma \ref{lem:PseudoMetric2}. Suppose in addition that $b$ is   componentwise nondecreasing. Then for all $\omega\in\Omega_2$
		\begin{align*} 
			\left|\varrho_{nk\kp}(0,x)(\omega)\right|\leq C_{2}(\omega)\sqrt{n}2^{-n}\Big(|x|+2^{-4^{n}}\Big)\,\text{ on }\Omega_{2}
		\end{align*}
		for all $x\in\mathbf{Q}$ and all choices of integers $n,\,k,\,\kp$ with $n\geq1$, $0\leq k,\kp\leq 2^n-1$.
	\end{lemma}
	}

\subsection{Main results and proofs}\label{sectmasinres}

In this section, we prove the path-by-path uniqueness of the solution to  \eqref{Eqmainhpde5}. We use this result to derive the existence and uniqueness of a strong solution to \eqref{Eqmainhpde5}. We assume the following conditions on the drift. We endow $\R^d$ with the partial order ``$\preceq$"  defined by
$$
x\preceq y\text{ when }x_i\leq y_i\text{ for all }i\in\{1,\ldots,d\}.
$$
\begin{hyp}\label{hyp1}\leavevmode
	\begin{enumerate}
		\item $b:\,\R_+^2\times\mathbb{R}^d\rightarrow \mathbb{R}^d$ is Borel measurable and admits the decomposition $b=\hat{b}-\check{b}$, where $\hat{b}(s,t,\cdot)$ and $\check{b}(s,t,\cdot)$ are   componentwise nondecreasing functions, that is each component $\hat{b}_i$ and $\check{b}_i$, $1\leq i\leq d$ is componentwise nondecreasing. Precisely, for every $x,y\in\R^d$,
		$$ x\preceq y\Rightarrow \hat{b}_i(s,t,x)\leq\hat{b}_i(s,t,y)\text{ and }\check{b}_i(s,t,x)\leq\check{b}_i(s,t,y).$$ 
		\item $b$ is of linear growth uniformly on $(s,t)$; precisely, there exists a positive constant $M$ such that
		\begin{align*}
			b(s,t,x)\leq M(1+|x|),\quad\forall\,(s,t,x)\in\R_+^2\times\R^d.
		\end{align*}
		
	\end{enumerate}
\end{hyp}

The main results of this section are the following : 

\begin{theorem}\label{mainres1}	Suppose $b$ satisfies Hypothesis \ref{hyp1}. Then the SDE \eqref{Eqmainhpde5} admits a unique strong solution. 
\end{theorem}

The above result constitutes an extension to those in \cite{BDM21b, NuTi97} by allowing the drift $b$ to be the difference of two monotone functions.  It is proved by using the Yamada-Watanabe principle. However, instead of showing the pathwise uniquess we show the following path-by-path uniqueness.
\begin{theorem}\label{maintheuniq1}
	Suppose $b$ satisfies Hypothesis \ref{hyp1}. Then for almost every Brownian sheet path $W$, \textcolor{black}{there exists} a unique continuous function $X :[0,1]^2\rightarrow \mathbb{R}^d$ satisfying \eqref{Eqmainhpde5}.
\end{theorem}
\textcolor{black}{The proof of Theorem \ref{maintheuniq1} is omitted since it follows the same lines as that of \cite[Theorem 3.2]{BDM21b}} provided that the next result holds.

\begin{theorem}\label{maintheuniq2}
	Suppose $b$ is as in Theorem \ref{maintheuniq1}. Suppose in addition that $b$ is uniformly bounded. Then for almost every Brownian sheet path $W$, \textcolor{black}{there exists} a unique continuous function $X :[0,1]^2\rightarrow \mathbb{R}^d$ satisfying \eqref{Eqmainhpde5}.
\end{theorem}

\begin{proof}[Proof of Theorem \ref{mainres1}]
	It follows from the  conditions of the theorem that, \eqref{Eqmainhpde5} has a weak solution. In addition, since path-by-path uniqueness implies pathwise uniqueness (see \cite[Page 9, Section 1.8.4]{BFGM14}), the result follows from the Yamada-Watanabe type principle for SDEs driven by Brownian sheets (see e.g. Nualart and Yeh \cite{NuYe89}).
\end{proof}

\begin{corollary}
	Suppose that $b$ is as in Theorem \ref{maintheuniq1}. Then for almost every Brownian sheet path $W$, \textcolor{black}{there exists} a unique continuous function $X :[0,1]^2\rightarrow \mathbb{R}^d$ satisfying \eqref{Eqmainhpde1}.
\end{corollary}

\begin{corollary}
	Suppose that $b$ is as in Theorem \ref{maintheuniq1}. Then for almost every Brownian sheet path $W$, \textcolor{black}{there exists} a unique continuous function $X :[0,1]^2\rightarrow \mathbb{R}^d$ satisfying the stochastic wave equation \eqref{Eqmainhpde3} when $a$ is the identity matrix.
\end{corollary}
\subsection{Proof of Theorem \ref{maintheuniq2}} 

In this subsection, we prove Theorem \ref{maintheuniq2}.  As already pointed out in the introduction, this is equivalent to showing that for almost all Brownian sheet path, the unique continuous solution $u$ to \eqref{Eqmainhpde6} is zero. More precisely, Theorem \ref{maintheuniq2} is equivalent to: 

\begin{theorem}\label{theo:DavieSheetMonotone}
	Let $W:=\left(W_{s,t},(s,t)\in[0,1]^2\right)$ be a $d$-dimensional Brownian sheet defined on a filtered probability space $(\Omega,\mathcal{F},\mathbb{F},\Pb)$, where $\mathbb{F}=\{\mathcal{F}_{s,t}\}_{s,t\in[0,1]}$. Let $b:\,[0,1]^2\times\R^d\to\R^d$ be a Borel measurable function such that for every $i\in\{1,\ldots,d\}$,  $b_i(s,t,\cdot)=\hat{b}_i(s,t,\cdot)-\check{b}_i(s,t,\cdot)$, where $\hat{b}_i$, $\check{b}_i$ are bounded and componentwise nondecreasing in $x$ for all $(s,t)$. 
	Then there exists $\Omega_1\subset\Omega$ with $\Pb(\Omega_1)=1$ such that for any $\omega\in\Omega_1$, $u=0$ is the unique continuous solution of the integral equation
	\begin{align}\label{eq:DavieSheetInt1D}
		u(s,t)=\int_0^t\int_0^s\left\{b(s_1,t_1,W_{s_1,t_1}(\omega)+u(s_1,t_1))-b(s_1,t_1,W_{s_1,t_1}(\omega))\right\}\mathrm{d}s_1\,\mathrm{d}t_1,\quad\forall\,(s,t)\in[0,1]^2.
	\end{align}
\end{theorem}
The proof of Theorem \ref{theo:DavieSheetMonotone} relies on the following Gronwall type result.
%Let $W$ be a Brownian sheet path for which the conclusions of Lemmas  \ref{lem:PseudoMetric1}, \ref{lem:PseudoMetric2} and \ref{lem:PseudoMetric3} are valid. Then
\begin{lemma}\label{lem:GronwallSheet}
	Suppose conditions of Theorem \ref{theo:DavieSheetMonotone} are valid. Then there exists $\Omega_1\subset\Omega$ with $\Pb(\Omega_1)=1$ and a positive random constant $C_2$ such that for any $\omega\in\Omega_1$, any sufficiently large positive integer $n$, any $(k,\kp)\in\{0,1,2,\cdots,2^n\}^2$, any $\beta(n)\in\left[2^{-4^{3n/4}},2^{-4^{2n/3}}\right]$, and any solution $u$ of the integral equation
	\begin{align}\label{eq:DavieSheetGronwall}
		&u(s,t)-u(s,0)-u(0,t)+u(0,0)\nonumber\\
		&=\int_0^t\int_0^s\Big\{b(s_1,t_1,W_{s_1,t_1}(\omega)+u(s_1,t_1))-b(s_1,t_1,W_{s_1,t_1}(\omega))\Big\}\mathrm{d}s_1\,\mathrm{d}t_1, 
		\quad\forall\,(s,t)\in[0,1]^2
	\end{align}
	satisfying
	\begin{equation}\label{eq:GronwallInitialV}
		\max\limits_{1\leq i\leq d}\max\{|u_i|(s,0),|u_i|(0,t)\}\leq\beta(n),\quad\forall\,(s,t)\in[0,1]^2,
	\end{equation}
	we have
	\begin{align}\label{eq:GronwallSheetEst}
		\max\limits_{1\leq i\leq d} \max\{\ovu_{n,i}(k,\kp),\und_{n,i}(k,\kp)\}\leq3^{k+\kp-1}\left(1+3C_2(\omega)\sqrt{dn}2^{-n}\right)^{k+\kp}\beta(n)\,\text{},
	\end{align}
	where $\ovu_{n}=(\ovu_{n,1},\ldots,\ovu_{n,d})$, $\und_{n}=(\und_{n,1},\ldots,\und_{n,d})$ and for every $i\in\{1,\ldots,d\}$, 
	\begin{align*}
		\overline{u}_{n,i}(k,\kp)=\sup\limits_{(s,t)\in I_{n,k-1}\times I_{n,\kp-1}}\max\{0,u_i(s,t)\}
		\text{ and }
		\underline{u}_{n,i}(k,\kp)=\sup\limits_{(s,t)\in I_{n,k-1}\times I_{n,\kp-1}}\max\{0,-u_i(s,t)\}.
	\end{align*}
\end{lemma}
\begin{proof}
	Suppose without loss of generality that $\Vert\hat{b}_i \Vert_{\infty}\leq1$ and $\Vert \check{b}_i\Vert_{\infty}\leq1$ for every $i\in\{1,\ldots,d\}$.
	By Lemma \ref{lem:PseudoMetric3}, there exists a subset $\Omega_{2}\subset \Omega$ with $\Pb(\Omega_{2})=1$  such that for all $\omega\in\Omega_2$,
	\begin{align}\label{eq:coPseudoMetricc31}
		\max\limits_{1\leq i\leq d}\left|\hat{\varrho}^{(i)}_{nk\kp}(0,x)(\omega)\right|\leq C_2(\omega)\sqrt{n}2^{-n}\left(|x|+\beta(n)\right)\,\text{ on }\Omega_{2}
	\end{align}
	and
	\begin{align}\label{eq:coPseudoMetricc32}
		\max\limits_{1\leq i\leq d}	\left|\check{\varrho}^{(i)}_{nk\kp}(0,x)(\omega)\right|\leq C_2(\omega)\sqrt{n}2^{-n}\left(|x|+\beta(n)\right)\,\text{ on }\Omega_{2}
	\end{align}
	for all integers $n,\,k,\,\kp$ with $n\geq1$, $0\leq k,\kp\leq 2^n-1$ and all $x\in[-1,1]^d$, where
	\begin{align*}
		\hat{\varrho}^{(i)}_{nk\kp}(0,x)(\omega)=\int_{k2^{-n}}^{(k+1)2^{-n}}\int_{\kp2^{-n}}^{(\kp+1)2^{-n}}\Big\{\hat{b}_i(s_1,t_1,W_{s_1,t_1}(\omega)+x)-\hat{b}_i(s_1,t_1,W_{s_1,t_1}(\omega))\Big\}\mathrm{d}s_1\,\mathrm{d}t_1
	\end{align*}
	and
	\begin{align*}
		\check{\varrho}^{(i)}_{nk\kp}(0,x)(\omega)=\int_{k2^{-n}}^{(k+1)2^{-n}}\int_{\kp2^{-n}}^{(\kp+1)2^{-n}}\Big\{\check{b}_i(s_1,t_1,W_{s_1,t_1}(\omega)+x)-\check{b}_i(s_1,t_1,W_{s_1,t_1}(\omega))\Big\}\mathrm{d}s_1\,\mathrm{d}t_1.
	\end{align*}
	For any $\omega\in\Omega_2$, we choose $n\in\N^{\ast}$ such that $C_2(\omega)\sqrt{dn}2^{-n}\leq1/6$ and split the set $[0,1]\times[0,1]$ onto $4^n$ squares $I_{nk}\times I_{n\kp}$. We set $u=(u_1,\ldots,u_d)$, with $u^+=(u^+_1,\ldots,u^+_d)$, $u^-=(u^-_1,\ldots,u^-_d)$, $u_i^{+}=\max\{0,u_i\}$ and $u_i^{-}=\max\{0,-u_i\}$ for every $i\in\{1,\ldots,d\}$.
	Since $\hat{b}_i(s_1,t_1,\cdot)$ and $\check{b}_i(s_1,t_1,\cdot)$ are nondecreasing, we deduce from \eqref{eq:DavieSheetGronwall} that for all $i\in\{1,\ldots,d\}$ and all $(s,t)\in I_{nk}\times I_{n\kp}$,
	\begin{align*}
		&u_i(s,t)-u_i(s,\kp2^{-n})-u_i(k2^{-n},t)+u_i(2^{-n}(k,\kp))\\
		=&\int_{k2^{-n}}^{s}\int_{\kp2^{-n}}^{t}\Big\{\hat{b}_i(s_1,t_1,W_{s_1,t_1}(\omega)+u(s_1,t_1))-\hat{b}_i(s_1,t_1,W_{s_1,t_1}(\omega))\Big\}\mathrm{d}s_1\,\mathrm{d}t_1\\
		&-\int_{k2^{-n}}^{s}\int_{\kp2^{-n}}^{t}\Big\{\check{b}_i(s_1,t_1,W_{s_1,t_1}(\omega)+u(s_1,t_1))-\check{b}_i(s_1,t_1,W_{s_1,t_1}(\omega))\Big\}\mathrm{d}s_1\,\mathrm{d}t_1\\
		\leq&\int_{k2^{-n}}^{s}\int_{\kp2^{-n}}^{t}\Big\{\hat{b}_i(s_1,t_1,W_{s_1,t_1}(\omega)+u^+(s_1,t_1))-\hat{b}_i(s_1,t_1,W_{s_1,t_1}(\omega))\Big\}\mathrm{d}s_1\,\mathrm{d}t_1\\
		&-\int_{k2^{-n}}^{s}\int_{\kp2^{-n}}^{t}\Big\{\check{b}_i(s_1,t_1,W_{s_1,t_1}(\omega)-u^{-}(s_1,t_1))-\check{b}_i(s_1,t_1,W_{s_1,t_1}(\omega))\Big\}\mathrm{d}s_1\,\mathrm{d}t_1. 
	\end{align*}
	Then, using the fact that $\max\{0,x+y\}\leq\max\{0,x\}+\max\{0,y\}$, we have
	\begin{align*}
		u_i^{+}(s,t)\leq& \max\{0,u_i(s,\kp2^{-n})+u_i(k2^{-n},t)-u_i(2^{-n}(k,\kp))\}\\
		&+\int_{k2^{-n}}^{s}\int_{\kp2^{-n}}^{t}\Big\{\hat{b}_i(s_1,t_1,W_{s_1,t_1}(\omega)+u^{+}(s_1,t_1))-\hat{b}_i(s_1,t_1,W_{s_1,t_1}(\omega))\Big\}\mathrm{d}s_1\,\mathrm{d}t_1\\
		&-\int_{k2^{-n}}^{s}\int_{\kp2^{-n}}^{t}\Big\{\check{b}_i(s_1,t_1,W_{s_1,t_1}(\omega)-u^{-}(s_1,t_1))-\check{b}_i(s_1,t_1,W_{s_1,t_1}(\omega))\Big\}\mathrm{d}s_1\,\mathrm{d}t_1.
	\end{align*}
	As a consequence,
	\begin{align}\label{eq:SheetEstuPlus}
		u_i^{+}(s,t)&\leq u_i^{+}(s,\kp2^{-n})+u_i^{+}(k2^{-n},t)+u_i^{-}(2^{-n}(k,\kp))
		+\hat{\varrho}^{(i)}_{nk\kp}\left(0,\ovu_n(k+1,\kp+1)\right)(\omega)\\&\qquad-\check{\varrho}^{(i)}_{nk\kp}\left(0,-\underline{u}_n(k+1,\kp+1)\right)(\omega)\nonumber
	\end{align}
	for all $(s,t)\in I_{nk}\times I_{n\kp}$.
	Similarly, we can show that
	\begin{align}\label{eq:SheetEstMinus}
		u_i^{-}(s,t)&\leq u_i^{-}(s,\kp2^{-n})+u_i^{-}(k2^{-n},t)+u_i^{+}(2^{-n}(k,\kp))
		-\hat{\varrho}^{(i)}_{nk\kp}\left(0,-\underline{u}_n(k+1,\kp+1)\right)(\omega)\\&\qquad+\check{\varrho}^{(i)}_{nk\kp}\left(0,\ovu_n(k+1,\kp+1)\right)(\omega)\nonumber
	\end{align}
	for all $(s,t)\in I_{nk}\times I_{n\kp}$.
	For any $k,\,\kp\in\{1,2,\cdots,2^n\}$ and $i\in\{1,\ldots,d\}$, we define $\widehat{u}_{n,i}(k,\kp)=\max\{\ovu_{n,i}(k,\kp),\und_{n,i}(k,\kp)\}$.
	\textcolor{black}{We deduce from Inequalities \eqref{eq:coPseudoMetricc31}-\eqref{eq:SheetEstMinus} that
		\begin{align*}
			\ovu_{n,i}(k+1,\kp+1)
			\leq& \max\limits_{1\leq j\leq d}\wdu_{n,j}(k,\kp+1)+\max\limits_{1\leq j\leq d}\wdu_{n,j}(k+1,\kp)+\max\limits_{1\leq j\leq d}\wdu_{n,j}(k,\kp)\\&+2C_2(\omega)\sqrt{n}2^{-n}\left(\sqrt{d}\max\limits_{1\leq j\leq d}\wdu_{n,j}(k+1,\kp+1)+\beta(n)\right)
		\end{align*}
		and
		\begin{align*}
			\und_{n,i}(k+1,\kp+1)
			\leq& \max\limits_{1\leq j\leq d}\wdu_{n,j}(k,\kp+1)+\max\limits_{1\leq j\leq d}\wdu_{n,j}(k+1,\kp)+\max\limits_{1\leq j\leq d}\wdu_{n,j}(k,\kp)\\&+2C_2(\omega)\sqrt{n}2^{-n}\left(\sqrt{d}\max\limits_{1\leq j\leq d}\wdu_{n,j}(k+1,\kp+1)+\beta(n)\right).
		\end{align*}
		Then
		\begin{align*}
			\max\limits_{1\leq i\leq d}\wdu_{n,i}(k+1,\kp+1)
			\leq& \max\limits_{1\leq i\leq d}\wdu_{n,i}(k,\kp+1)+\max\limits_{1\leq i\leq d}\wdu_{n,i}(k+1,\kp)+\max\limits_{1\leq i\leq d}\wdu_{n,i}(k,\kp)\\&+2C_2(\omega)\sqrt{n}2^{-n}\left(\sqrt{d}\max\limits_{1\leq i\leq d}\wdu_{n,i}(k+1,\kp+1)+\beta(n)\right).
		\end{align*}
		Since $C_2(\omega)\sqrt{dn}2^{-n}\leq1/6$, we have $(1-2C_2(\omega)\sqrt{dn}2^{-n})^{-1}\leq(1+3C_2(\omega)\sqrt{dn}2^{-n})$ and the above inequality implies
		\begin{align*}
			\max\limits_{1\leq i\leq d}\wdu_{n,i}(k+1,\kp+1)
			\leq& (1+3C_2(\omega)\sqrt{dn}2^{-n})\Big(\max\limits_{1\leq i\leq d}\wdu_{n,i}(k,\kp+1)+\max\limits_{1\leq i\leq d}\wdu_{n,i}(k+1,\kp)+\max\limits_{1\leq i\leq d}\wdu_{n,i}(k,\kp)\\&+2C_2(\omega)\sqrt{n}2^{-n} \beta(n)\Big).
		\end{align*}
		The desired result then follows by induction on $k$ and $\kp$ as in the proof of Lemma 3.9 in \cite{BDM21b}.}

\end{proof}

We now turn to the proof of  Theorem 
\ref{theo:DavieSheetMonotone}.

\begin{proof}[Proof of  Theorem \ref{theo:DavieSheetMonotone}]
	Choose $\Omega_1$, $\omega$, $n$ and $\beta(n)$ as in Lemma \ref{lem:GronwallSheet}.
	Let $u$ be a solution of (\ref{eq:DavieSheetInt1D}).  We have $\max\{|u|(s,0),|u|(0,t)\}=0\leq\beta(n)$ for all $(s,t)\in[0,1]^2$. Moreover, we deduce from (\ref{eq:GronwallSheetEst}) that
	\begin{align}\label{eq:GrowllUniqBound}
		\sup\limits_{k,\kp\in\{0,1,2,\cdots,2^n\}}\max\limits_{1\leq i\leq d}\max\{\ovu_{n,i}(k,\kp),\und_{n,i}(k,\kp)\}\leq 2^{2^{n+2}}\beta(n)
	\end{align}
	for all $n$ satisfying $C_2(\omega)\sqrt{dn}2^{-n}\leq1/9$. Since the right hand side of (\ref{eq:GrowllUniqBound}) converges to $0$ as $n$ goes to $\infty$, then, for all $(s,t)$, we have $u(s,t)=0$ on $\Omega_1$.
\end{proof}

\section{Malliavin regularity}\label{Malreg}

In this section we study the Malliavin regularity of the strong solution to \eqref{Eqmainhpde5}. 
\subsection{Basic facts on Malliavin calculus and compactness criterion on the plane}
We first recall some basic facts on Malliavin calculus for Wiener functionals on the plane which can be found in \cite[Section 2]{NuSa85} (see also \cite[Section 1]{NuSa89}).
Let $(\Omega,\mathcal{F},\Pb)$ be the canonical space associated to the $d$-dimensional Brownian sheet, that is $\Omega$ is the space of all continuous functions $\omega:\,\Gamma\to\R^d$ which vanish on the axes, $\Pb$ is the Wiener measure and $\mathcal{F}$ is the completion of the Borel $\sigma$-algebra of $\Omega$ with respect to $\Pb$. Let $(\mathcal{F}_{s,t},(s,t)\in\Gamma)$ denote the nondecreasing family of $\sigma$-algebras where $\mathcal{F}_{s,t}$ is generated by the functions $(s_1,t_1)\mapsto\omega(s_1\wedge s,t_1\wedge t)$, $(s_1,t_1)\in\Gamma$, $\omega\in\Omega$ and the null sets of $\mathcal{F}$. Consider the following subset $H$ of $\Omega$:
$$
H=\Big\{\omega\in\Omega:\,\text{there exists }\dot{\omega}\in L^2(\Gamma,\R^d)\text{ such that }\omega(s,t)=\int_0^s\int_0^t\dot{\omega}(s_1,t_1)\,\mathrm{d}t_1\mathrm{d}s_1,\text{ for any }(s,t)\in\Gamma\Big\}.
$$
Endowed with the inner product 
$$
\langle \omega_1,\omega_2\rangle_H=\sum\limits_{i=1}^d\int_{\Gamma}\dot{\omega}_1^{(i)}(s_1,t_1)\dot{\omega}_2^{(i)}(s_1,t_1)\mathrm{d}s_1\mathrm{d}t_1,
$$
the set $H$ is a Hilbert space. We call Wiener functional any measurable function defined on the Wiener space $(\Omega,\mathcal{F},\Pb)$. A Wiener functional $F:\,\Omega\to\R$ is said to be smooth if there exists some integer $n\geq1$ and an infinitely differentiable function $f$ on $\R^n$ such that 
\begin{enumerate}
	\item[(i)] $f$ and all its derivatives have at most polynomial growth order,
	\item[(ii)] $F(\omega)=f(\omega(s_1,t_1),\ldots,\omega(s_n,t_n))$ for some $(s_1,t_1),\ldots,(s_n,t_n)\in\Gamma$.
\end{enumerate}
Every smooth functional $F$ is Fr\'echet-differentiable and the Fr\'echet-derivative of $F$ along any vector $h\in H$ is given by
$$
DF(h)=\sum\limits_{j=1}^d\sum\limits_{i=1}^n\frac{\partial f}{\partial x_i^{(j)}}(\omega(s_1,t_1),\ldots,\omega(s_n,t_n))h^{(j)}(s_i,t_i)=\sum\limits_{j=1}^d\int_{\Gamma}\xi_j(q,r)\dot{h}^{(j)}(q,r)\mathrm{d}q\mathrm{d}r,
$$
where 
$$
\xi_j(q,r)=\sum\limits_{i=1}^n\frac{\partial f}{\partial x_i^{(j)}}(\omega(s_1,t_1),\ldots,\omega(s_n,t_n))\mathbf{1}_{[0,s_i]\times[0,t_i]}(q,r).
$$

Let $\mathbb{D}_{2,1}$ denote the closed hull of the family of smooth functionals with respect to the norm
$$
\Vert F\Vert^2_{2,1}=\Vert F\Vert^2_{L^2(\Omega)}+\Vert DF\Vert^2_{L^2(\Omega;H)}.
$$

Now we present a useful characterization of relatively compact subsets in the space $L^2(\Omega,\R^d)$. Let us recall the following compactness criterion provided in \cite[Theorem 1]{DaMN92}. 
\begin{theorem}\label{theoCompact}
	Let $A$ be a self-adjoint compact operator on $H$. Then, for any $c>0$, the set 
	$$
	\mathcal{G}=\left\{G\in\mathbb{D}_{2,1}:\,\Vert G\Vert_{L^2(\Omega)}+\Vert A^{-1}DG\Vert_{L^2(\Omega;H)}\leq c\right\}
	$$
	is relatively compact in $L^2(\Omega,\mathbb{R}^d)$.
\end{theorem} 
\noindent
In order to apply the above result, we consider the fractional Sobolev space: 
$$
\mathbb{G}^2_{\beta}(U;\R):=\Big\{g\in L^2(U,\R):\,\int_U\int_U\frac{\vert g(u)-g(u')\vert^2}{\vert u-u'\vert^{p+2\beta}}\mathrm{d}u\mathrm{d}u'<\infty\Big\},
$$ 
where $U$ is a domain of $\R^p$, $p\geq1$ and the norm is given by 
$$
\Vert g\Vert_{\mathbb{G}^2_{\beta}(U;\R)}:=\Vert g\Vert_{L^2(U;\R)}+\Big(\int_U\int_U\frac{\vert g(u)-g(u')\vert^2}{\vert u-u'\vert^{p+2\beta}}\mathrm{d}u\mathrm{d}u'\Big)^{1/2}.
$$
We need the next compact embedding result  from \cite[Lemma 10]{PSV13} (see also \cite[Theorem 7.1]{DNPV12}).
\begin{lemma}\label{lemCompact}
	Let $p\geq1$, $U\subset\R^p$ be a Lipschitz bounded open set and $\mathcal{J}$ be a bounded subset of $L^2(U;\R)$. Suppose that
	$$
	\sup\limits_{g\in\mathcal{J}}\int_U\int_U\frac{\vert g(u)-g(u')\vert^2}{\vert u-u'\vert^{p+2\beta}}\mathrm{d}u\mathrm{d}u'<\infty
	$$
	for some $\beta\in(0,1)$. Then $\mathcal{J}$ is relatively compact in $L^2(U;\R)$.
\end{lemma}
As a consequence of Theorem \ref{theoCompact} and Lemma \ref{lemCompact}, we have the following compactness criterion for subsets in the space $L^2(\Omega,\R^d)$. 
\begin{corollary}\label{CorolCompact}
	Denote by $\mathcal{F}^{W}_{\infty}$ the $\sigma$-algebra generated by the $d$-dimensional Brownian sheet $W=(W^{(1)},\ldots,W^{(d)})$. Let $(X^{(n)},n\in\N)$ be a sequence of $(\mathcal{F}^{W}_{\infty},\mathcal{B}(\R^d))$-measurable random variables and let $D_{s,t}$ be the Malliavin derivative associated with the random vector $W_{s,t}=(W_{s,t}^{(1)},\ldots,W_{s,t}^{(d)})$. Suppose 
	\begin{align}\label{EqSeqXn1}
		\sup\limits_{n\in\N}\Vert X^{(n)}\Vert_{L^2(\Omega,\R^d)}<\infty\,\text{ and }\,\sup\limits_{n\in\N}\Vert D_{\cdot,\cdot}X^{(n)}\Vert_{L^2(\Omega\times[0,1]^2,\R^{d\times d})}<\infty
	\end{align}
	as well as
	\begin{align}\label{EqSeqXn2}
		\sup\limits_{n\in\N}\int_0^1\int_0^1\int_0^1\int_0^1\frac{\E\left[\Vert D_{s,t}X^{(n)}-D_{\ps,\tp}X^{(n)}\Vert^2\right]}{(|s-\ps|+|t-\tp|)^{2+2\beta}}\mathrm{d}s\mathrm{d}\ps\mathrm{d}t\mathrm{d}\tp<\infty
	\end{align}
	for some $\beta\in(0,1)$. Then $(X^{(n)},n\in\N)$ is relatively compact in $L^2(\Omega,\R^d)$.
\end{corollary}
\begin{proof}
	The proof is inspired from \cite[Section 5]{HP14}. We consider the symmetric form $\mathcal{L}$ on $L^2((0,1)^2,\R^d)$ defined as
	\begin{align*}
		\mathcal{L}(f,g)=\int_0^1\int_0^1f(s,t)\cdot g(s,t)\mathrm{d}s\mathrm{d}t+\int_0^1\int_0^1\int_0^1\int_0^1\frac{(f(s,t)-f(\ps,\tp))\cdot(g(s,t)-g(\ps,\tp))}{(|s-\ps|+|t-\tp|)^{2+2\beta}}\mathrm{d}s\mathrm{d}\ps\mathrm{d}t\mathrm{d}\tp
	\end{align*}
	for functions $f$, $g$ in the dense domain $ \mathcal{D}(\mathcal{L})\subset L^2((0,1)^2,\R^d)$ and a fixed $\beta\in(0,1)$, where
	\begin{align*}
		\mathcal{D}(\mathcal{L})=\Big\{g:\,\Vert g\Vert^2_{L^2((0,1)^2,\R^d)}+\int_0^1\int_0^1\int_0^1\int_0^1\frac{|g(s,t)-g(\ps,\tp)|^2}{(|s-\ps|+|t-\tp|)^{2+2\beta}}\mathrm{d}s\mathrm{d}\ps\mathrm{d}t\mathrm{d}\tp<\infty\Big\}.
	\end{align*}
	Then $\mathcal{L}$ is a positive symmetric closed form and, by Kato's first representation theorem (see e.g. \cite{Ka13}), one can find a positive self-adjoint operator $T_{\mathcal{L}}$ such that
	\begin{align*}
		\mathcal{L}(f,g)=\langle f,T_{\mathcal{L}}g\rangle_{L^2((0,1)^2,\R^d)}
	\end{align*}
	for all $g\in\mathcal{D}(T_{\mathcal{L}})$ and $f\in\mathcal{D}(\mathcal{L})$. Further, one may observe that the form $\mathcal{L}$ is bounded from below by a positive number. Indeed, 
	\begin{align}\label{EqOpMatcalE}
		\mathcal{L}(g,g)\geq\Vert g\Vert_{L^2((0,1)^2,\R^d)}
	\end{align}
	for all $g\in\mathcal{D}(\mathcal{L})$. Hence, we have that $\mathcal{D}(\mathcal{L})=\mathcal{D}(T^{1/2}_{\mathcal{L}})$. 
	
	Now, define the operator $A$ as $A=T^{1/2}_{\mathcal{L}}$. It follows from Lemma 1 in \cite[Section 1]{Le82} (see also \cite[Lemma 9]{HP14}) and Lemma \ref{lemCompact} applied to $p=2$, $U=(0,1)^2$ that $A$ has a discrete spectrum and a compact inverse $A^{-1}$. Then, using \eqref{EqSeqXn1}-\eqref{EqOpMatcalE}, the operator $A$ and the sequence $(X^{(n)},n\in\N)$ satisfy the assumptions of Theorem \ref{theoCompact}.
\end{proof}

\subsection{Malliavin differentiability for bounded drifts}
In this subsection, we assume in addition the drift $b$ is bounded. The main result of this subsection which significantly generalised those in \cite{NuSa89} when the diffusion is constant is the following
\begin{theorem}\label{themmalldiff1}
	The strong solution $\{X^{\xi}_{s,t},\,s,t\in[0,1]\}$ of the SDE \eqref{Eqmainhpde5} is Malliavin differentiable. 
\end{theorem}

The proof of this theorem is done is two steps:

\textbf{Step 1:} We use standard approximation  procedure to approximate the drift coefficient $b=\hat{b}-\check{b}$ by a sequence of functions 
$$b_n:=\hat{b}_{n}-\check{b}_{n}, n\geq 1$$
such that $\hat{b}_{n}=(\hat{b}_{1,n},\ldots,\hat{b}_{d,n})$, $\check{b}_{n}=(\check{b}_{1,n},\ldots,\check{b}_{d,n})$, $(\hat{b}_{j,n})_{n\geq 1}$ and $(\check{b}_{j,n})_{n\geq 1}$ are smooth and componentwise non-decreasing functions with $\sup_n\|\hat{b}_{j,n}\|_\infty\leq \|\hat{b}_j\|_\infty<\infty $ and $\sup_n\|\check{b}_{j,n}\|_\infty\leq \|\check{b}_j\|_\infty<\infty $. In addition,  $(\hat{b}_{n})_{n\geq 1}$ (respctively, $(\check{b}_{n})_{n\geq 1}$) converges to  $\hat{b}$ (respctively, $\check{b}$) in $(s,t,x) \in [0,1]^2\times \R^d$ $\mathrm{d}s\times\mathrm{d}t\times\mathrm{d}x$-a.e. We know that for such smooth drift coefficients, the corresponding SDEs have a unique strong solution denoted by $X^{\xi,n}$. We then show that for each $s,t\in [0,1]$, the sequence $(X^{\xi,n}_{s,t})_{n \geq 1}$ is relatively compact in $L^2(\Omega,\Pb;\R^d)$.

\textbf{Step 2:} We show that the the sequence of solution $(X^{\xi,n})_{n \geq 1}$ converges strongly in $L^2(\Omega,\Pb;\R^d)$.

From \textbf{Step 1}  and \textbf{Step 2}, the result will follow by application of a compactness criteria \cite[Lemma 1.2.3]{Nu06}.

The next result corresponds to a $L^2(\Omega)$ compactness criteria. It is analogous to the result derived in  \cite{MMNPZ13} for the case of Brownian motion.
\begin{theorem}\label{thmcompact1}
	For every $(s,t)\in\Gamma_0$, the sequence $(X^{\xi,n}_{s,t})_{n \geq 1}$, is relatively compact in $L^2(\Omega,\Pb;\R^d)$.
\end{theorem}
The proof of the above theorem uses Corollary \ref{CorolCompact}, which in our case is reduced to proving:

\begin{lemma}\label{thmcompact2}
	There exists $C_1>0$ such that for every $(s,t)\in\Gamma_0$, the sequence $(X^{\xi,n}_{s,t} )_{n\geq1}$   satisfies
	\begin{align}\label{eqcompact3}
		\sup_{n\geq1} \Vert  X^{\xi,n}_{s,t}\Vert_{L^2(\Omega,\Pb;\R^d)}^2  \leq  C_1
	\end{align}
	and
	\begin{align}\label{eqcompact2}
		\sup_{n\geq1}\,	\sup_{
			\begin{subarray}{c}
				0 \leq r \leq s\\
				0 \leq u \leq t
		\end{subarray}} \E \left[ \Vert D_{r,u}X^{\xi,n}_{s,t}\Vert^2 \right] \leq  C_1.
	\end{align}
	Moreover for all $0 \leq r',r \leq s,\,\,0 \leq u',u \leq t$,
	\begin{align}\label{eqcompact1}
		\E \left[ \Vert D_{r,u}X^{\xi,n}_{s,t} - D_{r',u'}X^{\xi,n}_{s,t} \Vert^2 \right] \leq C_1(|r -r'|^{} +|u -u'|^{}),
	\end{align}
	where $\Vert\cdot\Vert$ denotes the max norm.
\end{lemma}
\begin{remark}
	If \eqref{eqcompact1} is satisfied, then for any $\beta\in(0,1/2)$, \eqref{EqSeqXn2} holds. Indeed, since, for any $s$, $\ps$, $t$, $\tp$, one has
	\begin{align*}
		|s-\ps|^{1/2}|t-\tp|^{1/2}\leq\frac{1}{2}(|s-\ps|+|t-\tp|),
	\end{align*}
	then for any $n\in\N$,
	\begin{align*}
		&\int_0^1\int_0^1\int_0^1\int_0^1\frac{\E\left[\Vert D_{s,t}X^{(n)}-D_{\ps,\tp}X^{(n)}\Vert^2\right]}{(|s-\ps|+|t-\tp|)^{2+2\beta}}\mathrm{d}s\mathrm{d}\ps\mathrm{d}t\mathrm{d}\tp\\\leq& C_1\int_0^1\int_0^1\int_0^1\int_0^1\frac{\mathrm{d}s\mathrm{d}\ps\mathrm{d}t\mathrm{d}\tp}{(|s-\ps|+|t-\tp|)^{1+2\beta}}\leq\dfrac{C_1}{2}\Big(\int_0^1\int_0^1\frac{\mathrm{d}s\mathrm{d}\ps}{|s-\ps|^{1/2+\beta}}\Big)\Big(\int_0^1\int_0^1\frac{\mathrm{d}t\mathrm{d}\tp}{|t-\tp|^{1/2+\beta}}\Big)<\infty.
	\end{align*}
\end{remark}
\noindent
To prove the result above, we need some preliminary estimates.
\begin{lemma}\label{lemMalEstimate1}
	There exists a non-decreasing function $\widetilde{C}_1:\,\R_+\times\R_+\to\R_+$ such that for any $0< r<s\leq1$, any $0< u<t\leq1$, any $k\in\R_+$ and any $i,j\in\{1,\cdots,d\}$,
	\begin{align}\label{eqDavie1}
		\mathbb{E}\Big[\exp\Big(\frac{k}{\delta(r,s)\delta(u,t)} \int_r^s\int_u^t\partial_i\hat{b}_{j,n}(s_1,t_1,W_{s_1,t_1})\mathrm{d}t_1\mathrm{d}s_1\Big)\Big]\leq \widetilde{C}_1(k,\Vert \hat{b}_{j,n}\Vert_{\infty})
	\end{align}
	and
	\begin{align}\label{eqDavie2}
		\mathbb{E}\Big[\exp\Big(\frac{k}{\delta(r,s)\delta(u,t)} \int_r^s\int_u^t\partial_i\check{b}_{j,n}(s_1,t_1,W_{s_1,t_1})\mathrm{d}t_1\mathrm{d}s_1\Big)\Big]\leq \widetilde{C}_1(k,\Vert \check{b}_{j,n}\Vert_{\infty}),
	\end{align}
	where $\delta(r,s)=\sqrt{s-r}$ and $\delta(u,t)=\sqrt{t-u}$.
\end{lemma}
\begin{proof} 
	We only prove \eqref{eqDavie1} and the proof of \eqref{eqDavie2} follows  anagously. Using integration with respect to local time formula (see e.g. \cite[Corollary 2.3]{BDM21a}), we have
	\begin{align}
		&\int_r^{s}\int_u^{t}\partial_i\hat{b}_{j,n}(s_1,t_1,W_{s_1,t_1})\mathrm{d}t_1\mathrm{d}s_1\nonumber\\
		=&-\int_r^{s}\int_u^{t}b_{j,n}(s_1,t_1,W_{s_1,t_1})\frac{\mathrm{d}_{t_1}W^{(i)}_{s_1,t_1}}{s_1}\mathrm{d}s_1+\int_r^{s}\int_{1-t}^{1-u}b_{j,n}(s_1,1-t_1,W_{s_1,1-t_1})\frac{\mathrm{d}_{t_1}B^{(i)}_{s_1,t_1}}{s_1}\mathrm{d}s_1\label{EqEisenMall1}\\
		&\nonumber-\int_r^{s}\int_{1-t}^{1-u}b_{j,n}(s_1,1-t_1,W_{s_1,1-t_1})\frac{W^{(i)}_{s_1,1-t_1}}{s_1(1-t_1)}\mathrm{d}t_1\mathrm{d}s_1, 
	\end{align}
	where $(B^{(i)}_{s,t},(s,t)\in[0,1]^2)$ is the Brownian sheet given by the following  representation provided by Dalang and Walsh \cite{DaW02}
	\begin{align*}
		W^{(i)}_{s,1-t}=W^{(i)}_{s,1}+B^{(i)}_{s,t}-\int_0^t\frac{W^{(i)}_{s,1-u}}{1-u}\mathrm{d}u.
	\end{align*} 
	Since the function $x\longmapsto e^{3x}$ is convex,
	\begin{align}\label{IneqDWest1}
		&\mathbb{E}\Big[\exp\Big(\frac{k}{\delta(r,s)\delta(u,t)}\Big|\int_r^s\int_u^t\partial_i\hat{b}_{j,n}(s_1,t_1,W_{s_1,t_1})\mathrm{d}t_1\mathrm{d}s_1\Big|\Big)\Big]\nonumber\\
		\leq&\frac{1}{3}\Big\{\mathbb{E}\Big[\exp\Big(\frac{3k}{\delta(r,s)\delta(u,t)}\Big|\int_r^{s}\int_u^{t}\hat{b}_{j,n}(s_1,t_1,W_{s_1,t_1})\frac{\mathrm{d}_{t_1}W^{(i)}_{s_1,t_1}}{s_1}\mathrm{d}s_1\Big|\Big)\Big]\\
		&+ \mathbb{E}\Big[\exp\Big(\frac{3k}{\delta(r,s)\delta(u,t)}\Big|\int_r^{s}\int_{1-t}^{1-u}\hat{b}_{j,n}(s_1,t_1,W_{s_1,1-t_1})\frac{\mathrm{d}_{t_1}B^{(i)}_{s_1,t_1}}{s_1}\mathrm{d}s_1 \Big|\Big)\Big]\nonumber\\
		&+\mathbb{E}\Big[\exp\Big(\frac{3k}{\delta(r,s)\delta(u,t)}\Big|\int_r^{s}\int_{1-t}^{1-u}\hat{b}_{j,n}(s_1,t_1,W_{s_1,1-t_1})\frac{W^{(i)}_{s_1,1-t_1}}{s_1(1-t_1)}\mathrm{d}t_1\mathrm{d}s_1\Big|\Big)\Big]\Big\}=\frac{1}{3}(I_1+I_2+I_3).\nonumber
	\end{align}
	By Jensen inequality
	\begin{align*}
		I_1=&\mathbb{E}\Big[\exp\Big(\frac{3k}{\delta(r,s)\delta(u,t)}\Big|\int_r^{s}\int_u^{t}\hat{b}_{j,n}(s_1,t_1,W_{s_1,t_1})\frac{\mathrm{d}_{t_1}W^{(i)}_{s_1,t_1}}{s_1}\mathrm{d}s_1\Big|\Big)\Big]\\\leq&\int_r^s\mathbb{E}\Big[\exp\Big(\frac{6k(\sqrt{s}-\sqrt{r})}{\delta(r,s)\delta(u,t)}\Big|\int_u^{t}\hat{b}_{j,n}(s_1,t_1,W_{s_1,t_1})\frac{\mathrm{d}_{t_1}W^{(i)}_{s_1,t_1}}{\sqrt{s_1}}\Big|\Big)\Big] \frac{\mathrm{d}s_1}{2\sqrt{s_1}(\sqrt{s}-\sqrt{r})}\\
		\leq&\int_r^s\mathbb{E}\Big[\exp\Big(\frac{6k}{\delta(u,t)}\Big|\int_u^{t}\hat{b}_{j,n}(s_1,t_1,W_{s_1,t_1})\frac{\mathrm{d}_{t_1}W^{(i)}_{s_1,t_1}}{\sqrt{s_1}}\Big|\Big)\Big] \frac{\mathrm{d}s_1}{2\sqrt{s_1}(\sqrt{s}-\sqrt{r})}.
	\end{align*}	
	Since for every $s_1\in[r,s]$, $$\Big(Y_{s_1,t_1}:=\int_{u}^{t_1}\hat{b}_{j,n}(s_1,t_2,W_{s_1,t_2})\frac{\mathrm{d}_{t_2}W^{(i)}_{s_1,t_2}}{\sqrt{s_1}},u\leq t_1\leq t\Big)$$ is a square integrable martingale, it follows from the Barlow-Yor inequality that there exists a positive constant $c_1$ such that for any positive integer $m$ and any $s_1\in[r,s]$,
	\begin{align*}
		\E[|Y_{s_1,t}|^m] \leq&\E\Big[\sup\limits_{u\leq t_1\leq t}|Y_{s_1,t_1}|^m\Big]\leq c_1^m\sqrt{m}^{\,m}\E[\langle Y_{s_1,\cdot}\rangle^{m/2}_t]=c_1^m\sqrt{m}^{\,m}\E\Big[\Big(\int_{u}^{t}\hat{b}^2_{j,n}(s_1,t_1,W_{s_1,t_1})\,\mathrm{d}t_1\Big)^{m/2}\Big]\\ \leq&c_1^m\sqrt{m}^{\,m}\delta(u,t)^{m}\Vert \hat{b}_{j,n}\Vert^m_{\infty}.
	\end{align*}
	Hence, using the following exponential expansion formula,
	\begin{align*}
		\exp\Big(\frac{6k}{\delta(u,t)}|Y_{s_1,t}|\Big)=\sum\limits_{m=0}^{\infty}\frac{6^mk^m|Y_{s_1,t}|^m}{\delta(u,t)^mm!},
	\end{align*}
	we obtain
	\begin{align*}
		I_1\leq&\int_r^s\mathbb{E}\Big[\exp\Big(\frac{6k}{\delta(u,t)}\Big|\int_u^{t}\hat{b}_{j,n}(s_1,t_1,W_{s_1,t_1})\frac{\mathrm{d}_{t_1}W_{s_1,t_1}}{\sqrt{s_1}}\Big|\Big)\Big] \frac{\mathrm{d}s_1}{2\sqrt{s_1}(\sqrt{s}-\sqrt{r})}\\=&\int_r^s\E\Big[\exp\Big(\frac{6k}{\delta(u,t)}|Y_{s_1,t}|\Big)\Big]\frac{\mathrm{d}s_1}{2\sqrt{s_1}(\sqrt{s}-\sqrt{r})}=\int_r^s\sum\limits_{m=0}^{\infty}\frac{6^mk^m\E[|Y_{s_1,t}|^m]}{\delta(u,t)^mm!}\,\frac{\mathrm{d}s_1}{2\sqrt{s_1}(\sqrt{s}-\sqrt{r})}\\ \leq&\sum\limits_{m=0}^{\infty}\frac{6^mk^mc_1^m\sqrt{m}^{\,m}\Vert \hat{b}_{j,n}\Vert^m_{\infty}}{m!}:=\widetilde{C}_{1,1}(k,\Vert \hat{b}_{j,n}\Vert_{\infty}),
	\end{align*}
	which is finite by ratio test. Similarly, we also have
	\begin{align*}
		I_2\leq\widetilde{C}_{1,1}(k,\Vert \hat{b}_{j,n}\Vert_{\infty}).
	\end{align*}
	To estimate $I_3$ we apply one more Jensen inequality and obtain
	\begin{align*}
		I_3=& \mathbb{E}\Big[\exp\Big(\frac{3k}{\delta(r,s)\delta(u,t)}\Big|\int_r^{s}\int_{1-t}^{1-u}\hat{b}_{j,n}(s_1,t_1,W_{s_1,1-t_1})\frac{W^{(i)}_{s_1,1-t_1}}{s_1(1-t_1)}\mathrm{d}t_1\mathrm{d}s_1\Big|\Big)\Big]\\
		\leq&\int_r^{s}\int_{1-t}^{1-u}\mathbb{E}\Big[\exp\Big(\frac{12k(\sqrt{s}-\sqrt{r})(\sqrt{1-u}-\sqrt{1-t})}{\delta(r,s)\delta(u,t)}|\hat{b}_{j,n}(s_1,t_1,W_{s_1,1-t_1})|\Big|\frac{W^{(i)}_{s_1,1-t_1}}{\sqrt{s_1}\sqrt{1-t_1}}\Big|\Big)\Big]\\&\qquad\qquad\times\frac{\mathrm{d}t_1}{2\sqrt{1-t_1}(\sqrt{1-u}-\sqrt{1-t})}\frac{\mathrm{d}s_1}{2\sqrt{s_1}(\sqrt{s}-\sqrt{r})}\\
		\leq&\int_r^{s}\int_{1-t}^{1-u}\mathbb{E}\Big[\exp\Big(12k\Vert \hat{b}_{j,n}\Vert_{\infty} \Big|\frac{W^{(i)}_{s_1,1-t_1}}{\sqrt{s_1}\sqrt{1-t_1}}\Big|\Big)\Big]\frac{\mathrm{d}t_1}{2\sqrt{1-t_1}(\sqrt{1-u}-\sqrt{1-t})}\frac{\mathrm{d}s_1}{2\sqrt{s_1}(\sqrt{s}-\sqrt{r})}\\
		\leq&2\exp\Big(72k^2\Vert \hat{b}_{j,n}\Vert^2_{\infty}\Big):=\widetilde{C}_{1,2}(k,\Vert \hat{b}_{j,n}\Vert_{\infty}),
	\end{align*}
	since $\Big|\frac{W^{(i)}_{s_1,1-t_1}}{\sqrt{s_1}\sqrt{1-t_1}}\Big|$ is a reduced Gaussian random variable.\\ The proof is completed by choosing $\widetilde{C}_1=2\widetilde{C}_{1,1}+\widetilde{C}_{1,2}$.
\end{proof} 

\begin{proof}[Proof of Lemma \ref{thmcompact2}] 
	Without loss of generality, we suppose $\xi=0$. Further we set $X^n_{s,t}:=X^{0,n}_{s,t}$. We start with the proof of \eqref{eqcompact2}. Suppose  for every $j\in\{1,\ldots,d\}$, $\hat{b}_{j,n}$ (respectively, $\check{b}_{j,n}$) is componentwise nondecreasing and continuously differentiable with bounded derivatives.
	Then for any $(r,u)\in[0,1]^2$ with $(0,0)\prec(r,u)\prec(s,t)$, 
	\begin{align}\label{eqmalder1}
		D_{r,u}X^n_{s,t}=I_d+\int_r^s\int_u^t\nabla b_n^{}(s_1,t_1,X^n_{s_1,t_1})\,D_{r,u}X^n_{s_1,t_1}\mathrm{d}t_1\mathrm{d}s_1,
	\end{align}
	where $I_d$ is the identity matrix, $\nabla b_n=(\partial_ib_{j,n})_{1\leq i,j\leq d}$, $\partial_ib_{j,n}$ being the partial derivative of $b_j(s,t,\cdot)$ with respect to $x_i$,
	and $D_{r,u}X_{s,t}^n=(D_{r,u}^iX_{s,t}^{n,j})_{1\leq i,j\leq d}$. Since $\partial_i\hat{b}_{j,n}$ and $\partial_i\check{b}_{j,n}$ are \textcolor{black}{nonnegative}, we have
	\begin{align*}
		\Vert D_{r,u}X^n_{s,t}\Vert\leq1+\int_r^s\int_u^t\sum\limits_{i,j=1}^d\Big\{\partial_i\hat{b}_{j,n}^{}(s_1,t_1,X^n_{s_1,t_1})+\partial_i\check{b}_{j,n}^{}(s_1,t_1,X^n_{s_1,t_1})\Big\}\Vert D_{r,u}X^n_{s_1,t_1}\Vert\mathrm{d}t_1\mathrm{d}s_1.
	\end{align*} 
	Therefore (see e.g. \cite[Lemma 5.1.1]{Qi16}),
	\begin{align}\label{eqmalder11}
		\Vert D_{r,u}X^n_{s,t}\Vert\leq&\exp\Big(\int_r^s\int_u^t\sum\limits_{i,j=1}^d\Big\{\partial_i\hat{b}_{j,n}^{}(s_1,t_1,X^n_{s_1,t_1})+\partial_i\check{b}_{j,n}^{}(s_1,t_1,X^n_{s_1,t_1})\Big\}\mathrm{d}t_1\mathrm{d}s_1\Big). 
	\end{align}
	Squaring both sides of the inequality, taking the expectation and using the Girsanov theorem (see e.g. \cite[Theorem 3.5]{DM15} and \cite[Proposition 1.6]{NP94}) and the Cauchy-Schwarz inequality, we obtain
	\begin{align*}
		&\mathbb{E}[\Vert D_{r,u}X^n_{s,t}\Vert^2]\\
		\leq&\E\Big[\exp\Big(2\int_r^s\int_u^t\sum\limits_{i,j=1}^d\Big\{\partial_i\hat{b}_{j,n}^{}(s_1,t_1,X^n_{s_1,t_1})+\partial_i\check{b}_{j,n}^{}(s_1,t_1,X^n_{s_1,t_1})\Big\}\mathrm{d}t_1\mathrm{d}s_1\Big)\Big]\\=&\E\Big[\mathcal{E}\Big(\int_0^1\int_0^1b_n(s_1,t_1,W_{s_1,t_1})\cdot\mathrm{d}W_{s_1,t_1}\Big)\exp\Big(2\int_r^s\int_u^t\sum\limits_{i,j=1}^d\Big\{\partial_i\hat{b}_{j,n}^{}(s_1,t_1,W_{s_1,t_1})+\partial_i\check{b}_{j,n}^{}(s_1,t_1,W_{s_1,t_1})\Big\}\mathrm{d}t_1\mathrm{d}s_1\Big)\Big]\\ \leq& C_0\E\Big[\exp\Big(4\int_r^s\int_u^t\sum\limits_{i,j=1}^d\Big\{\partial_i\hat{b}_{j,n}^{}(s_1,t_1,X^n_{s_1,t_1})+\partial_i\check{b}_{j,n}^{}(s_1,t_1,X^n_{s_1,t_1})\Big\}\mathrm{d}t_1\mathrm{d}s_1\Big)\Big]^{\frac{1}{2}}\\=&C_0\E\Big[\prod\limits_{i,j=1}^d\exp\Big(4\int_r^s\int_u^t\partial_i\hat{b}_{j,n}(s_1,t_1,W_{s_1,t_1})\mathrm{d}t_1\mathrm{d}s_1\Big)\exp\Big(4\int_r^s\int_u^t\partial_i\check{b}_{j,n}(s_1,t_1,W_{s_1,t_1})\mathrm{d}t_1\mathrm{d}s_1\Big)\Big]^{\frac{1}{2}},
	\end{align*}
	where
	\begin{align}\label{EqUnifL2BigEcalBound}
		C_0:=\sup\limits_{n\geq1}\E\Big[\mathcal{E}\Big(\int_0^1\int_0^1b_n(s_1,t_1,W_{s_1,t_1})\cdot\mathrm{d}W_{s_1,t_1}\Big)^2\Big]^{\frac{1}{2}}
	\end{align}
	is finite (see Lemma \ref{lemUnifL2ExpBound} and Remark \ref{remUnifL2ExpBound} in the Appendix), with
	$$
	\mathcal{E}\Big(\int_0^1\int_0^1b_n(s_1,t_1,W_{s_1,t_1})\cdot\mathrm{d}W_{s_1,t_1}\Big)=\exp\Big(\int_0^1\int_0^1b_n(s_1,t_1,W_{s_1,t_1})\cdot\mathrm{d}W_{s_1,t_1}-\frac{1}{2}\int_0^1\int_0^1|b_n(s_1,t_1,W_{s_1,t_1})|^2\cdot\mathrm{d}s_1\mathrm{d}t_1\Big).
	$$ 
	Hence, by H\"older inequality, we have
	\begin{align*}%\label{eqmalderprof1}
		&\mathbb{E}[\Vert D_{r,u}X^n_{s,t}\Vert^2]\\
		\leq& C_0\prod\limits_{i,j=1}^d\mathbb{E}\Big[\exp\Big(8d^2\int_r^s\int_u^t\partial_i\hat{b}_{j,n}(s_1,t_1,W_{s_1,t_1})\mathrm{d}t_1\mathrm{d}s_1\Big)\Big]^{\frac{1}{4d^2}}\mathbb{E}\Big[\exp\Big(8d^2\int_r^s\int_u^t\partial_i\check{b}_{j,n}(s_1,t_1,W_{s_1,t_1})\mathrm{d}t_1\mathrm{d}s_1\Big)\Big]^{\frac{1}{4d^2}}.\notag
	\end{align*}
	Then, by Lemma \ref{lemMalEstimate1}, we obtain
	\begin{align*}
		\mathbb{E}[\Vert D_{r,u}X^n_{s,t}\Vert^2]\leq C_0\times\widetilde{C}_1\Big(8d^2,\max\limits_{1\leq j\leq d}\{\Vert \hat{b}_{j,n}\Vert_{\infty}+\Vert \check{b}_{j,n}\Vert_{\infty}\}\Big)\leq C_0\times\widetilde{C}_1\Big(8d^2,\Vert \hat{b}\Vert_{\infty}+\Vert \check{b}\Vert_{\infty}\Big),
	\end{align*}
	which means that the Malliavin derivative of $X^n$ is bounded in $L^2(\Omega,\Pb;\R^d)$. %We have to show that $X^n$ converges strongly to $X$ in $L^2$.
	%\redsout{We can show using our previous results that this bound is finite and only depends on $\|b_{1,n}\|_{\infty}\leq \|b_1\|_{\infty}\leq \|b\|_{\infty}$. Similarly we can show that the second term on the right side of \eqref{eqmalderprof1} is bounded by $ \|b\|_{\infty}$. Therefore the we conclude that} 

	Next we prove \eqref{eqcompact1}.	We deduce from \eqref{eqmalder1} that for all $0 \leq r' \leq r \leq s\leq 1,\,\,0 \leq u' \leq u \leq t\leq1$, 
	
	\begin{align*}
		&D_{r,u}X^n_{s,t}- D_{r',u'}X^n_{s,t}\\
		=&\int_r^s\int_u^t\nabla b_n^{}(s_1,t_1,X^n_{s_1,t_1})\,D_{r,u}X^n_{s_1,t_1}\,\mathrm{d}t_1\mathrm{d}s_1-\int_{r'}^s\int_{u'}^t\nabla b_n^{}(s_1,t_1,X^n_{s_1,t_1})\,D_{r',u'}X^n_{s_1,t_1}\,\mathrm{d}t_1\mathrm{d}s_1\\
		=& \int_r^s\int_u^t\nabla b_n^{}(s_1,t_1,X^n_{s_1,t_1})\Big(D_{r,u}X^n_{s_1,t_1}-D_{r',u'}X^n_{s_1,t_1}\Big)\mathrm{d}t_1\mathrm{d}s_1-\int_{r'}^s\int_{u'}^u\nabla b_n^{}s_1,t_1,X^n_{s_1,t_1})\,D_{r',u'}X^n_{s_1,t_1}\,\mathrm{d}t_1\mathrm{d}s_1\\
		&-\int_{r'}^r\int_{u}^t\nabla b_n^{}(s_1,t_1,X^n_{s_1,t_1})\,D_{r',u'}X^n_{s_1,t_1}\,\mathrm{d}t_1\mathrm{d}s_1.
	\end{align*}	
	Taking the absolute value on both side and using the fact that $b_n=\hat{b}_n-\check{b}_n$, with $\hat{b}_n,\check{b}_n$ nondecreasing, gives
	\begin{align*}
		&\Big\Vert D_{r,u}X^n_{s,t}- D_{r',u'}X^n_{s,t}\Big\Vert\\
		\leq& \int_r^s\int_u^t\sum\limits_{i,j=1}^d\Big\{\partial_i\hat{b}_{j,n}^{}(s_1,t_1,X^n_{s_1,t_1})+\partial_i\check{b}_{j,n}^{}(s_1,t_1,X^n_{s_1,t_1})\Big\}\Big\Vert D_{r,u}X^n_{s_1,t_1}-D_{r',u'}X^n_{s_1,t_1}\Big\Vert\mathrm{d}t_1\mathrm{d}s_1\\
		&+\int_{r'}^s\int_{u'}^u\sum\limits_{i,j=1}^d\Big\{\partial_i\hat{b}_{j,n}^{}(s_1,t_1,X^n_{s_1,t_1})+\partial_i\check{b}_{j,n}^{}(s_1,t_1,X^n_{s_1,t_1})\Big\}\Big\Vert D_{r',u'}X^n_{s_1,t_1}\Big\Vert\mathrm{d}t_1\mathrm{d}s_1\\
		&+\int_{r'}^r\int_{u}^t\sum\limits_{i,j=1}^d\Big\{\partial_i\hat{b}_{j,n}^{}(s_1,t_1,X^n_{s_1,t_1})+\partial_i\check{b}_{j,n}^{}(s_1,t_1,X^n_{s_1,t_1})\Big\}\Big\Vert D_{r',u'}X^n_{s_1,t_1}\Big\Vert\mathrm{d}t_1\mathrm{d}s_1.	
	\end{align*}
	Applying \cite[Lemma 5.1.1]{Qi16}, we obtain
	\begin{align}\label{eqmallDerr1}
		&\Big\Vert D_{r,u}X^n_{s,t}- D_{r',u'}X^n_{s,t}\Big\Vert
		\\
		%\leq &  \int_r^s\int_u^t\Big(b_{1,n}^{\prime}(X^n_{s_1,t_1})+b_{2,n}^{\prime}(X^n_{s_1,t_1}\Big)\Big|D_{r,u}X^n_{s_1,t_1}-D_{r',u'}X^n_{s_1,t_1}\Big|\mathrm{d}t_1\mathrm{d}s_1\\
		%		&+2\int_{r'}^s\int_{u'}^t\Big(b_{1,n}^{\prime}(X^n_{s_1,t_1})+b_{2,n}^{\prime}(X^n_{s_1,t_1}\Big)\Big|D_{r',u'}X^n_{s_1,t_1}\Big|\mathrm{d}t_1\mathrm{d}s_1\\
		\leq& \Big(\int_{r'}^s\int_{u'}^u\sum\limits_{i,j=1}^d\Big\{\partial_i\hat{b}_{j,n}^{}(s_1,t_1,X^n_{s_1,t_1})+\partial_i\check{b}_{j,n}^{}(s_1,t_1,X^n_{s_1,t_1})\Big\}\Big\Vert D_{r',u'}X^n_{s_1,t_1}\Big\Vert\mathrm{d}t_1\mathrm{d}s_1\notag\\&
		+\int_{r'}^r\int_{u}^t\sum\limits_{i,j=1}^d\Big\{\partial_i\hat{b}_{j,n}^{}(s_1,t_1,X^n_{s_1,t_1})+\partial_i\check{b}_{j,n}^{}(s_1,t_1,X^n_{s_1,t_1})\Big\}\Big\Vert D_{r',u'}X^n_{s_1,t_1}\Big\Vert\mathrm{d}t_1\mathrm{d}s_1\Big)\notag\\
		&\times		\exp\Big(\int_r^s\int_u^t\sum\limits_{i,j=1}^d\Big\{\partial_i\hat{b}_{j,n}^{}(s_1,t_1,X^n_{s_1,t_1})+\partial_i\check{b}_{j,n}^{\prime}(s_1,t_1,X^n_{s_1,t_1})\Big\}\mathrm{d}t_1\mathrm{d}s_1\Big)\notag
	\end{align}	
	Since $\partial_i\hat{b}_{j,n}^{}$ and $\partial_i\check{b}_{j,n}^{}$ are nonnegative, it follows from \eqref{eqmalder11} that
	\begin{align*}
		\Vert D_{r',u'}X^n_{s_1,t_1}\Vert\leq \exp\Big(\int_{r'}^{s}\int_{u'}^{t}\sum\limits_{i,j=1}^d\Big\{\partial_i\hat{b}_{j,n}^{}(s_{2},t_2,X^n_{s_2,t_2})+\partial_i\check{b}_{j,n}^{}(s_2,t_2,X^n_{s_2,t_2})\Big\}\mathrm{d}t_2\mathrm{d}s_2\Big).
	\end{align*}
	Hence, we deduce from  \eqref{eqmallDerr1} that it holds: 	
	\begin{align*}
		\Big\Vert D_{r,u}X^n_{s,t}- D_{r',u'}X^n_{s,t}\Big\Vert
		\leq&  \Big(\int_{r'}^s\int_{u'}^u\sum\limits_{i,j=1}^d\Big\{\partial_i\hat{b}_{j,n}^{}(s_1,t_1,X^n_{s_1,t_1})+\partial_i\check{b}_{j,n}^{}(s_1,t_1,X^n_{s_1,t_1})\Big\}\mathrm{d}t_1\mathrm{d}s_1
		\\&\qquad+\int_{r'}^r\int_{u}^t\sum\limits_{i,j=1}^d\Big\{\partial_i\hat{b}_{j,n}^{}(s_1,t_1,X^n_{s_1,t_1})+\partial_i\check{b}_{j,n}^{}(s_1,t_1,X^n_{s_1,t_1})\Big\}\mathrm{d}t_1\mathrm{d}s_1\Big)\\
		&\times\exp\Big(2\int_{r'}^s\int_{u'}^t\sum\limits_{i,j=1}^d\Big\{\partial_i\hat{b}_{j,n}^{}(s_1,t_1,X^n_{s_1,t_1})+\partial_i\check{b}_{j,n}^{}(s_1,t_1,X^n_{s_1,t_1})\Big\}\mathrm{d}t_1\mathrm{d}s_1\Big).\notag
	\end{align*}
	
	Squaring both sides of the inequality, taking the expectation and using Cauchy-Schwarz inequality and Girsanov theorem give
	\begin{align*}
		&\mathbb{E}\Big[\Big\Vert D_{r,u}X^n_{s,t}- D_{r',u'}X^n_{s,t}\Big\Vert^2\Big]\\
		\leq&
		4^7C_0\Big(\sum\limits_{i,j=1}^d\Big\{\E\Big[\Big(\int_{r'}^s\int_{u'}^u\partial_i\hat{b}_{j,n}^{}(s_1,t_1,W_{s_1,t_1})\,\mathrm{d}t_1\mathrm{d}s_1\Big)^8\Big]+\E\Big[\Big(\int_{r'}^s\int_{u'}^u\partial_i\check{b}_{j,n}^{}(s_1,t_1,W_{s_1,t_1})\,\mathrm{d}t_1\mathrm{d}s_1\Big)^8\Big]\\
		&\qquad\quad+\E\Big[\Big(\int_{r'}^r\int_{u}^t\partial_i\hat{b}_{j,n}^{}(s_1,t_1,W_{s_1,t_1})\,\mathrm{d}t_1\mathrm{d}s_1\Big)^8\Big]+\E\Big[\Big(\int_{r'}^r\int_{u}^t\partial_i\check{b}_{j,n}^{}(s_1,t_1,W_{s_1,t_1})\,\mathrm{d}t_1\mathrm{d}s_1\Big)^8\Big]\Big\}\Big)^{\frac{1}{4}}\\
		&\times\E\Big[\exp\Big(16\int_{r'}^s\int_{u'}^t\sum\limits_{i,j=1}^d\Big\{\partial_i\hat{b}_{j,n}^{}(s_1,t_1,W_{s_1,t_1})+\partial_i\check{b}_{j,n}^{}(s_1,t_1,W_{s_1,t_1})\Big\}\mathrm{d}t_1\mathrm{d}s_1\Big)\Big]^{\frac{1}{4}}\\
		=&4^7C_0(J_1+J_2+J_3+J_4)^{1/4}\times J_5,
	\end{align*}
	where $C_0$ is given by \eqref{EqUnifL2BigEcalBound}.
	It follows from H\"older inequality and the estimates \eqref{eqDavie1}-\eqref{eqDavie2} that 
	\begin{align*}
		J_5=&\E\Big[\exp\Big(16\int_{r'}^s\int_{u'}^t\sum\limits_{i,j=1}^d\Big\{\partial_i\hat{b}_{j,n}^{}(s_1,t_1,W_{s_1,t_1})+\partial_i\check{b}_{j,n}^{}(s_1,t_1,W_{s_1,t_1})\Big\}\mathrm{d}t_1\mathrm{d}s_1\Big)\Big]^\frac{1}{4}\\
		\leq& \prod\limits_{i,j=1}^d\mathbb{E}\Big[\exp\Big(32d^2\int_r^s\int_u^t\partial_i\hat{b}_{j,n}(s_1,t_1,W_{s_1,t_1})\mathrm{d}t_1\mathrm{d}s_1\Big)\Big]^{\frac{1}{8d^2}}\mathbb{E}\Big[\exp\Big(32d^2\int_r^s\int_u^t\partial_i\check{b}_{j,n}(s_1,t_1,W_{s_1,t_1})\mathrm{d}t_1\mathrm{d}s_1\Big)\Big]^{\frac{1}{8d^2}}
		\\
		\leq& \widetilde{C}_1(32d^2,\Vert\hat{b}\Vert_{\infty}+\Vert\check{b}\Vert_{\infty})^\frac{1}{4}.
	\end{align*}	 
	Moreover, using the inequality $x^8\leq8!e^x$ ($x\in\R$), we have
	\begin{align*}
		&\mathbb{E}\Big[\Big(\int_{r'}^s\int_{u'}^u \partial_i\hat{b}_{j,n}^{}(s_1,t_1,W_{s_1,t_1})\mathrm{d}t_1\mathrm{d}s_1\Big)^8\Big]+\mathbb{E}\Big[\Big(\int_{r'}^s\int_{u'}^u \partial_i\check{b}_{j,n}^{}(s_1,t_1,W_{s_1,t_1})\mathrm{d}t_1\mathrm{d}s_1\Big)^8\Big]\\
		\leq&8!\delta(r',s)^8\delta(u',u)^8\mathbb{E}\Big[\exp\Big(\frac{1}{\delta(r',s)\delta(u',u)}\Big|\int_{r'}^s\int_{u'}^u\partial_i\hat{b}_{j,n}^{}(s_1,t_1,W_{s_1,t_1})\mathrm{d}t_1\mathrm{d}s_1\Big|\Big)\Big]\\
		&\quad+8!\delta(r',s)^8\delta(u',u)^8\mathbb{E}\Big[\exp\Big(\frac{1}{\delta(r',s)\delta(u',u)}\Big|\int_{r'}^s\int_{u'}^u\partial_i\check{b}_{j,n}^{}(s_1,t_1,W_{s_1,t_1})\mathrm{d}t_1\mathrm{d}s_1\Big|\Big)\Big]\\
		\leq&2(8!)\widetilde{C}_1(1,\Vert \hat{b}\Vert_{\infty}+\Vert\check{b}\Vert_{\infty})\delta(u',u)^4
	\end{align*}
	and
	\begin{align*}
		&\mathbb{E}\Big[\Big(\int_{r'}^r\int_{u}^t \partial_i\hat{b}_{j,n}^{}(s_1,t_1,W_{s_1,t_1})\mathrm{d}t_1\mathrm{d}s_1\Big)^8\Big]+\mathbb{E}\Big[\Big(\int_{r'}^r\int_{u}^t \partial_i\check{b}_{j,n}^{}(s_1,t_1,W_{s_1,t_1})\mathrm{d}t_1\mathrm{d}s_1\Big)^8\Big]\\
		\leq&8!\delta(r',r)^8\delta(u,t)^8\mathbb{E}\Big[\exp\Big(\frac{1}{\delta(r',r)\delta(u,t)}\Big|\int_{r'}^r\int_u^t\partial_i\hat{b}_{j,n}^{}(s_1,t_1,W_{s_1,t_1})\mathrm{d}t_1\mathrm{d}s_1\Big|\Big)\Big]\\
		&\quad+8!\delta(r',r)^8\delta(u,t)^8\mathbb{E}\Big[\exp\Big(\frac{1}{\delta(r',r)\delta(u,t)}\Big|\int_{r'}^r\int_u^t\partial_i\check{b}_{j,n}^{}(s_1,t_1,W_{s_1,t_1})\mathrm{d}t_1\mathrm{d}s_1\Big|\Big)\Big]\\
		\leq&2(8!)\widetilde{C}_1(1,\Vert \hat{b}\Vert_{\infty}+\Vert \check{b}\Vert_{\infty})\delta(r,r')^8.
	\end{align*}
	As a consequence,
	\begin{align*}
		J_1+J_2+J_3+J_4\leq 4(8!)\widetilde{C}_1(1,\Vert \hat{b}\Vert_{\infty}+\Vert \check{b}\Vert_{\infty}) (|r-r'|^4+|u-u'|^4).
	\end{align*}
	Therefore
	\begin{align*}
		\mathbb{E}\Big[\Big\Vert D_{r,u}X^n_{s,t}- D_{r',u'}X^n_{s,t}\Big\Vert^2\Big]\leq 4^8(8!)C_0\times\widetilde{C}_1(32d^2,\Vert \hat{b}\Vert_{\infty}+\Vert\check{b}\Vert_{\infty})^{1/2}(|r-r'|+|u-u'|).
	\end{align*}
	Finally, by the Girsanov theorem and the Cauchy-Schwarz inequality, we have 
	\begin{align*}
		\sup\limits_{n\geq1}\|X^n_{s,t}\|^2_{L^2(\Omega,\Pb;\R^d)}=\sup\limits_{n\geq1}\E\left[\vert X_{s,t}^{n}\vert^2\right]&=\sup\limits_{n\geq1}\E\Big[\mathcal{E}\Big(\int_0^1\int_0^1b_n(s_1,t_1,W_{s_1,t_1})\cdot\mathrm{d}W_{s_1,t_1}\Big)|W_{s,t}|^2\Big]\\ &\leq\sup\limits_{n\geq1}\E\Big[\mathcal{E}\Big(\int_0^1\int_0^1b_n(s_1,t_1,W_{s_1,t_1})\cdot\mathrm{d}W_{s_1,t_1}\Big)^2\Big]^\frac{1}{2}\E\Big[|W_{s,t}|^4\Big]^{\frac{1}{2}}\leq C_0\sqrt{3d},
	\end{align*}
	where $C_0$ is given by \eqref{EqUnifL2BigEcalBound}. 
	
	The proof is completed by taking $C_1=C_0\max\{\sqrt{3d},4^8(8!)\widetilde{C}_1(32d^2,\Vert\hat{b}\Vert_{\infty}+\Vert\check{b}\Vert_{\infty})\}$.
\end{proof}

For $q\geq 1$, let us consider the following space $L^{q} (\mathbb{R}^d; \mathfrak{p}(x)\diffns x)$ defined by 
\begin{align}
	L^{q} (\mathbb{R}^d; \mathfrak{p}(x)\diffns x)=\Big\{h:\mathbb{R}^d\rightarrow \mathbb{R}^d \text{ measurable and such that }\int_{\mathbb{R}^d}|h(x)|^q\mathfrak{p}(x)\diffns x<\infty\Big\},
\end{align}
where the weight function $\mathfrak{p}(x)$ is defined by 
$$\mathfrak{p}(x)=e^{\frac{-|x|^2}{2st}},\,\,x\in \mathbb{R}^d.$$

\begin{theorem}\label{thml2con1}
	Let $b_n$ be defined as before and let $(X^{\xi,n})_{n\geq 1}$ be the sequence of corresponding strong solutions to the SDE  \eqref{Eqmainhpde5}. Then for any fixed $s,t\in [0,1]$, $(X_{s,t}^{\xi,n})_{n\geq 1}$ converges strongly in $L^2(\Omega,\Pb;\R^d)$ to $X_{s,t}^{\xi}$.
\end{theorem}
In order to prove the above theorem we need the subsequent result.
\begin{lemma}\label{thml2con2}
	Let $(X^{\xi,n})_{n\geq 1}$ be the sequence of corresponding strong solutions as given before.  Then for every $s,t \in [0,1]$ and function $h\in L^{4} (\mathbb{R}^d; \mathfrak{p}(x)\diffns x)$, it holds that the sequence $(h(X_{s,t}^{\xi,n}))_{n\geq 1}$ converges weakly in $L^2(\Omega,\Pb;\R^d)$ to $h(X_{s,t}^{\xi})$. 
\end{lemma}
\begin{proof}[Proof of Theorem \ref{thml2con1}]
	Using Theorem \ref{thmcompact1}, we know that for each $s,t$, there exists a subsequence $(X^{\xi,n_k}_{s,t})_{k\geq1}$ that converges strongly in $L^2(\Omega,\Pb;\R^d)$.  Set $h(x)=x,\, x\in \mathbb{R}^d$ and use Lemma \ref{thml2con2} to obtain that $(X^{\xi,n_k}_{s,t})_{n\geq1}$ converges weakly to $X^{\xi}_{s,t}$ in $L^2(\Omega,\Pb;\R^d)$. Thanks to the uniqueness of the limit, there exists a subsequence $n_k$ such that $(X^{\xi,n_k}_{s,t})_{n\geq1}$ converges strongly to $X^{\xi}_{s,t}$ in $L^2(\Omega,\Pb;\R^d)$. The convergence then holds for the entire sequence by uniqueness of the limit. To see this, suppose by contradiction that for some $s,t$, there exist $\epsilon >0$ and a subsequence $n_l, l\geq 0$ such that  
	$$
	\|X_{s,t}^{\xi,n_l}-X^{\xi}_{s,t}\|_{L^2(\Omega,\Pb;\R^d)}\geq \epsilon.
	$$
	We also know from the compactness criteria that there exists a further subsequence $n_m, m\ge 0$ of $n_l, l\geq 0$ such that 
	$$
	X_{s,t}^{\xi,n_{n_m}} \rightarrow \tilde{X}_{s,t}\text{ in } L^2(\Omega,\Pb;\R^d) \text{ as } m \rightarrow\infty .
	$$
	However, $(X^{\xi,n_k}_{s,t})_{n\geq1} \rightarrow X^{\xi}_{s,t}$ as $k\rightarrow \infty$ weakly in $L^2(\Omega,\Pb;\R^d)$, and hence by the uniqueness of the limit, we obtain
	$$
	\tilde{X}_{s,t}=X^{\xi}_{s,t}.
	$$ 
	Since 
	$$
	\|X_{s,t}^{\xi,n_{n_m}}-X^{\xi}_{s,t}\|_{L^2(\Omega,\Pb;\R^d)}\geq \epsilon,
	$$
	we obtain a contradiction.
\end{proof}
\begin{proof}[Proof of Theorem \ref{themmalldiff1}]
	We know from Theorem \ref{thml2con1} that $(X_{s,t}^{\xi,n})_{n\geq1}$ converges strongly in $L^2(\Omega,\Pb;\R^d)$ to $X^{\xi}_{s,t}$ and from \eqref{eqcompact2} in Lemma \ref{thmcompact2} that $(D_{r,u}X^{\xi,n}_{s,t})_{n\geq 1}$ is bounded in the $L^2([0,1]^2\times \Omega,\mathrm{d}s\times\mathrm{dt}\times\Pb;\R^{d\times d})$-norm uniformly in $n$. Thefore, using \cite[Lemma 1.2.3]{Nu06}, we also have that the limit $X^{\xi}_{s,t}$ is Malliavin differentiable.
\end{proof}

\begin{proof}[Proof of Lemma \ref{thml2con2}]
	%	Let us first show that $(h(X^{\xi,n}_{t}))_{n\geq 1}$ is bounded in 
	%	$L^2(\Omega,\Pb;\R^d)$. In fact, using Girsanov transform and H\"older inequality, we have
	%	\begin{align}\label{eqweaklim1}
	%		\sup_{n}	\mathbb{E}\left[|h(X^{\xi,n}_{s,t})|^2\right]&=\sup_{n}\E\left[ \mathcal{E}\Big(\int_0^1\int_0^1b_{n}(s_1,t_1,\xi+ W_{s_1,t_1})\cdot\mathrm{d} W_{s_1,t_1}\Big)|h(\xi+W_{s,t})|^2\right]\\ &\leq  \sup_{n}\mathbb{E}\left[ \mathcal{E}\Big(\int_0^1\int_0^1b_{n}(s_1,t_1,\xi
	%		+W_{s_1,t_1})\cdot\mathrm{d} W_{s_1,t_1}\Big)^2\right]^{\frac{1}{2}}\mathbb{E}\left[|h(\xi+  W_{s,t})|^4\right]^{\frac{1}{2}}\notag\\
	%		%&\quad \times E\left[e^{2\sum_{i=1}^d\int_0^1u_{i,n}^2(r,\omega,x+\sigma \cdot B_r)\diffns r}\right]^{\frac{1}{4}}\notag\\
	%		&\leq  C \E\left[|h(\xi+W_{s,t})|^4\right]^{\frac{1}{4}}\notag\\
	%		&= C\Big(\frac{1}{\sqrt{2\pi st}}\int_{\mathbb{R}^d}|h(\xi+x)|^4e^{\frac{-|x|^2}{2st}}\diffns x\Big)^{\frac{1}{4}}<\infty,\notag
	%	%	&\leq  \frac{C_{\|\sigma \|^2}}{(2\pi t\|\sigma \|^2)^{\frac{1}{8}}}\Big(|x|^{4}\int_{\mathbb{R}}e^{-\frac{|z|^2}{ 2t\|\sigma \|^2}}\diffns z +\int_{\mathbb{R}}e^{-\frac{|z|^2}{ 2^{5}t\|\sigma \|^2}}\diffns z \Big)^{\frac{1}{4}}<\infty.
	%	\end{align}
	%	where $\cdot$ denotes the usual scalar product in $\R^d$ and the constant $$C:=\sup_{n}\mathbb{E}\left[ \mathcal{E}\Big(\int_0^1\int_0^1b_{n}(s_1,t_1,\xi
	%	+W_{s_1,t_1})\cdot\mathrm{d} W_{s_1,t_1}\Big)^2\right]^{\frac{1}{2}}$$ depends only on $ \|\hat{b}\|_\infty$ and $\|\check{b}\|_{\infty}$.	
	%	
	%	We show that $(h(X^{\xi,n}_{t}))_{n\geq 1}\rightarrow h(X^{\xi}_t)$  weakly in $L^2(\Omega,\Pb;\R^d)$, by 
	Let us first noticing that the space
	$$
	\left\{\mathcal{E}\Big(\int_0^1\int_0^1\frac{\partial^2\varphi_{s_1,t_1}}{\partial s_1\partial t_1}\cdot\diffns W_{s_1,t_1}\Big): \varphi\in C_{2,b}([0,1]^2,\mathbb{R}^d)\right\}
	$$
	is a dense subspace of $L^2(\Omega,\Pb;\R^d)$. Here $C_{2,b}([0,1],\mathbb{R}^d)$ is the space of bounded vector functions $\varphi$ such that each component $\varphi^i$ has a second partial derivative $\frac{\partial^2\varphi^i_{s_1,t_1}}{\partial s_1\partial t_1}$ of bounded variation with values in $\mathbb{R}$.  %continuously differentiable functions on $[0,1]^2$ and with values in $\mathbb{R}$ and $\dot\varphi$ is the derivative of $\varphi$. 
	Hence, it suffices to show that for every $i$, $$\mathbb{E}\Big[h_i(X^{\xi,n}_{s,t})\mathcal{E}\Big(\int_0^1\int_0^1\frac{\partial^2\varphi_{s_1,t_1}}{\partial s_1\partial t_1}\cdot\diffns W_{s_1,t_1}\Big)\Big]\longrightarrow \mathbb{E}\Big[h_i(X^{\xi}_{s,t})\mathcal{E}\Big(\int_0^1\int_0^1\frac{\partial^2\varphi_{s_1,t_1}}{\partial s_1\partial t_1}\cdot\diffns W_{s_1,t_1}\Big)\Big] \text{ as } n\rightarrow \infty.$$ %converges to $E\Big[h(X^{\xi}_t)|\mathcal{F}_t\Big]\mathcal{E}\Big(\int_0^1\dot\varphi_r\diffns B_r\Big)$ in expectation. % 
	\textcolor{black}{
		Since $\Omega$ is a Wiener space, then, as in \cite[proof Lemma 2]{Kita51} or \cite[proof of Theorem 2]{Yeh63}, one can show a multidimensional analog of the Cameron-Martin translation theorem. Precisely for every $g:\,\R^d\to\R$ measurable, one has
		\begin{align}
			\label{eq:CMTranslaThm}
			\mathbb{E}\Big[g(X^{\xi}_{s,t})\mathcal{E}\Big(\int_0^1\int_0^1\frac{\partial^2\varphi_{s_1,t_1}}{\partial s_1\partial t_1}\cdot\diffns W_{s_1,t_1}\Big)\Big]=\int_{\Omega} g(X^{\xi}_{s,t}(\omega+\varphi))\diffns \mathbb{P}(\omega).
		\end{align}
	}
	Let $\varphi \in C_{2,b}([0,1]^2,\mathbb{R}^d)$.			
	For every $n$, the process $\tilde X^{\xi,n}$ given by $\tilde X^{\xi,n}_{s,t}(\omega):= X^{\xi,n}_{s,t}(\omega+ \varphi)$ solves the SDE
	\begin{equation}
		\label{eq:CM sde}
		\diffns \tilde X^{\xi,n}_{s,t} = \Big(b_{n}(t,\tilde X^{\xi,n}_{s,t}) +\frac{\partial^2\varphi_{s,t}}{\partial s\partial t}\Big)\diffns s\diffns t + \diffns W_{s,t}.
	\end{equation}
	
	%	To see this, let $H \in L^2(\Omega, P)$ and apply \eqref{eq:CM} and the fact that $X^{x,n}$ solves the SDE \eqref{eqmainr} to get 
	%	\begin{align*}
	%	E[\tilde X^{x,n}_tH] &= E\left[X^{x,n}_tH(\omega - \varphi)\mathcal{E}\left(\int_0^1\dot\varphi(u)\diffns B_u\right)\right]\\
	%	&= E\Big[\Big(x + \int_0^tb_1(u, X^{x,n}_u) + b_2(u)\diffns u + \sigma B_t \Big)H(\omega-\varphi)\mathcal{E}\Big(\int_0^1\dot\varphi\diffns B \Big) \Big] \\
	%	& = E\Big[\Big(x + \int_0^t b_1(u, X^{x,n}_u(\omega+ \varphi))+ b_2(\omega+\varphi)\diffns u  + \sigma B_t(\omega+\varphi)\Big)H\Big]\\
	%	&= E\Big[\Big(x + \int_0^t b_1(u, \tilde X^{x,n}_u(\omega))+ \tilde b_2(\omega)+ \sigma\dot\varphi\diffns u  + \sigma B_t(\omega)\Big)H\Big],
	%	\end{align*} 
	%	where the last equality follows by the fact that $B_t(\omega + \varphi) = B_t(\omega) + \varphi$, since $B$ is the canonical process.	This proves the claim.
	Since $X^{\xi}$ is also the solution to the SDE it holds that $\tilde X^{\xi}(\omega):= X^{\xi}(\omega+ \varphi)$ satisfies
	\begin{equation}
		\diffns \tilde X^{\xi}_{s,t} = \Big(b(t,\tilde X^{\xi}_{s,t}) +\frac{\partial^2\varphi_{s,t}}{\partial s\partial t}\Big)\diffns s\diffns t + \diffns W_{s,t},\quad \mathbb{P}\text{-a.s.}
	\end{equation}
	Applying \eqref{eq:CMTranslaThm} and the Girsanov theorem, we have
	\begin{align}\label{eqweaklim2}
		&	\mathbb{E}\Big[h_i(X^{\xi,n}_{s,t})\mathcal{E}\Big(\int_0^1\int_0^1\frac{\partial^2\varphi_{s_1,t_1}}{\partial s_1\partial t_1}\cdot\diffns W_{s_1,t_1}\Big)-h_i(X^{\xi}_{s,t})\mathcal{E}\Big(\int_0^1\int_0^1\frac{\partial^2\varphi_{s_1,t_1}}{\partial s_1\partial t_1}\cdot\diffns W_{s_1,t_1}\Big)\Big]\notag\\
		=&\mathbb{E}\Big[\Big(h_i(X^{\xi,n}_{s,t})-h_i(X^{\xi}_{s,t})\Big)\mathcal{E}\Big(\int_0^1\int_0^1\frac{\partial^2\varphi_{s_1,t_1}}{\partial s_1\partial t_1}\cdot\diffns W_{s_1,t_1}\Big)\Big]\notag\\
		=&  \mathbb{E}\Big[h_i(\xi+W_{s,t})\Big\{\mathcal{E}\Big(\int_0^1\int_0^1\Big\{b_n(s_1,t_1,\xi+W_{s_1,t_1})+\frac{\partial^2\varphi_{s_1,t_1}}{\partial s_1\partial t_1}\Big\}\cdot\diffns W_{s_1,t_1}\Big)\notag\\
		&-\mathcal{E}\Big(\int_0^1\int_0^1\Big\{b(s_1,t_1,\xi+W_{s_1,t_1})+\frac{\partial^2\varphi_{s_1,t_1}}{\partial s_1\partial t_1}\Big\}\cdot\diffns W_{s_1,t_1}\Big)\Big\}\Big].
	\end{align}
	Using the fact that $|e^a -e^b|\leq |e^a + e^b||a - b|$, the H\"older inequality and Burkholder-Davis-Gundy inequality, we get
	\begin{align}\label{eqweaklim3}
		&	\mathbb{E}\Big[h_i(X^{\xi,n}_{s,t})\mathcal{E}\Big(\int_0^1\int_0^1\frac{\partial^2\varphi_{s_1,t_1}}{\partial s_1\partial t_1}\cdot\diffns W_{s_1,t_1}\Big)-h_i(X^{\xi}_{s,t})\mathcal{E}\Big(\int_0^1\int_0^1\frac{\partial^2\varphi_{s_1,t_1}}{\partial s_1\partial t_1}\cdot\diffns W_{s_1,t_1}\Big)\Big]\notag\\
		\leq& C \mathbb{E}\Big[h_i(x+ W_{s,t})^2\Big]^{\frac{1}{2}}\mathbb{E}\Big[\Big(\mathcal{E}\Big(\int_0^1\int_0^1\Big\{b_n(s_1,t_1,\xi+W_{s_1,t_1})+\frac{\partial^2\varphi_{s_1,t_1}}{\partial s_1\partial t_1}\Big\}\cdot\diffns W_{s_1,t_1}\Big)\notag\\
		&+\mathcal{E}\Big(\int_0^1\int_0^1\Big\{b(s_1,t_1,\xi+W_{s_1,t_1})+\frac{\partial^2\varphi_{s_1,t_1}}{\partial s_1\partial t_1}\Big\}\cdot\diffns W_{s_1,t_1}\Big)\Big)^4\Big]^{\frac{1}{4}}\notag\\
		& \times\Big\{\mathbb{E}\Big[ \Big(\int_0^1\int_0^1\Big(b_n(s_1,t_1,\xi+W_{s_1,t_1})- b(s_1,t_1,\xi+W_{s_1,t_1})\Big)\cdot\diffns W_{s_1,t_1}\Big)^4\Big]\notag\\
		& +\mathbb{E}\Big[\Big(\textcolor{black}{\frac{1}{2}}\int_0^1\int_0^1\Big(b_n(s_1,t_1,\xi+W_{s_1,t_1})+\frac{\partial^2\varphi_{s_1,t_1}}{\partial s_1\partial t_1}\Big)^2	-\Big(b(s_1,t_1,\xi+W_{s_1,t_1})+\frac{\partial^2\varphi_{s_1,t_1}}{\partial s_1\partial t_1}\Big)^2\diffns s_1\diffns t_1\Big)^4\Big] \Big\}^{\frac{1}{4}}\notag\\
		\leq&\textcolor{black}{ C \mathbb{E}\Big[h_i(x+ W_{s,t})^2\Big]^{\frac{1}{2}}\mathbb{E}\Big[\Big(\mathcal{E}\Big(\int_0^1\int_0^1\Big\{b_n(s_1,t_1,\xi+W_{s_1,t_1})+\frac{\partial^2\varphi_{s_1,t_1}}{\partial s_1\partial t_1}\Big\}\cdot\diffns W_{s_1,t_1}\Big)\notag}\\
		& \textcolor{black}{+\mathcal{E}\Big(\int_0^1\int_0^1\Big\{b(s_1,t_1,\xi+W_{s_1,t_1})+\frac{\partial^2\varphi_{s_1,t_1}}{\partial s_1\partial t_1}\Big\}\cdot\diffns W_{s_1,t_1}\Big)\Big)^4\Big]^{\frac{1}{4}}\notag}\\
		&  \textcolor{black}{\times\Big\{ \int_0^1\int_0^1\mathbb{E}\Big[\Big|b_n(r,\xi+W_{s_1,t_1})- b(s_1,t_1,\xi+W_{s_1,t_1})\Big|^4\Big]\diffns s_1\diffns t_1\notag}\\
		& \textcolor{black}{ +\frac{1}{16}\int_0^1\int_0^1\mathbb{E}\Big[\Big|\Big(b_n(s_1,t_1,\xi+W_{s_1,t_1})+\frac{\partial^2\varphi_{s_1,t_1}}{\partial s_1\partial t_1}\Big)^2	-\Big(b(s_1,t_1,\xi+W_{s_1,t_1})+\frac{\partial^2\varphi_{s_1,t_1}}{\partial s_1\partial t_1}\Big)^2\Big|^4\Big]\diffns s_1\diffns t_1 \Big\}^{\frac{1}{4}}\notag}\\
		&=I_1\times I_{2,n}\times (I_{3,n}+I_{4,n}).
	\end{align}
	$I_1$ is finite since $h\in L^{4} (\mathbb{R}; \mathfrak{p}(x)\diffns x)$. Next observe that
	\begin{align*}
		&\mathcal{E}\Big(\int_0^1\int_0^1\Big\{b_n(s_1,t_1,\xi+W_{s_1,t_1})+\frac{\partial^2\varphi_{s_1,t_1}}{\partial s_1\partial t_1}\Big\}\cdot\diffns W_{s_1,t_1}\Big)\\
		=& \mathcal{E}\Big(\int_0^1\int_0^1b_n(s_1,t_1,\xi+W_{s_1,t_1})\cdot\diffns W_{s_1,t_1}\Big)\mathcal{E}\Big(\int_0^1\int_0^1\frac{\partial^2\varphi_{s_1,t_1}}{\partial s_1\partial t_1}\cdot\diffns W_{s_1,t_1}\Big)\\
		&\times\exp\Big(\int_0^1\int_0^1b_n(s_1,t_1,\xi+W_{s_1,t_1})\frac{\partial^2\varphi_{s_1,t_1}}{\partial s_1\partial t_1}\diffns  s_1\diffns  t_1\Big).
	\end{align*}
	Using the boundedness of $\frac{\partial^2\varphi_{s_1,t_1}}{\partial s_1\partial t_1}$ and the uniform boundedness of $b_n$, it follows that $ I_{2,n}$ is bounded.
	Using the dominated convergence theorem, we get that $I_{3,n}$ and $I_{4,n}$ converge to $0$ as $n$ goes to infinity. Let us for example consider the term  $I_{3,n}$. Using the density of the Brownian sheet for every $q\geq 1$ it holds:
	\begin{align*}
		\mathbb{E}\Big[\Big|b_n(s,t,\xi+W_{s,t})- b(s,t,\xi+W_{s,t})\Big|^q\Big]=&\frac{1}{\sqrt{2\pi st}}\int_{\mathbb{R}}|b_n(s,t,\xi+z)-b(s,t,\xi+z)|^qe^{\frac{-|z|^2}{2st}}\diffns z\\
		=&\frac{1}{\sqrt{2\pi st}}\int_{\mathbb{R}}|b_n(s,t,z)-b(s,t,z)|^qe^{\frac{-|z-\xi|^2}{2st}}\diffns z\\
		=&\frac{1}{\sqrt{2\pi st}}\int_{\mathbb{R}}|b_n(s,t,z)-b(s,t,z)|^qe^{\frac{-|z-2\xi|^2}{4st}}e^{\frac{-|z|^2}{4st}}e^{\frac{|\xi|^2}{4st}}\diffns z\\
		\leq&\frac{e^{\frac{|\xi|^2}{4st}}}{\sqrt{2\pi st}}\int_{\mathbb{R}}|b_n(s,t,z)-b(s,t,z)|^qe^{\frac{-|z|^2}{4st}}\diffns z.
	\end{align*}
	Thus the result follows by the dominated convergence theorem.
\end{proof}

\subsection{Malliavin regularity under linear growth condition}
As in the previous section, we approximate the drift coefficient $b=\hat{b}-\check{b}$ by a sequence of functions 
$b_n:=\hat{b}_{n}-\check{b}_{n}, n\geq 1$,
where $\hat{b}_{n}=(\hat{b}_{1,n},\ldots,\hat{b}_{d,n})$, $\check{b}_{n}=(\check{b}_{1,n},\ldots,\check{b}_{d,n})$, $(\hat{b}_{j,n})_{n\geq 1}$ and $(\check{b}_{j,n})_{n\geq 1}$ are smooth, componentwise non-decreasing and bounded functions satisfying:
\begin{enumerate}
	\item[$\bullet$]There exists $\tilde{M}>0$ such that $\|\hat{b}_{j,n}(s,t,x)\|\leq \tilde{M}(1+|x|)$ and $\|\hat{b}_{j,n}(s,t,x)\|\leq \tilde{M}(1+|x|)$ for all $n\geq1$ and $(s,t,x)\in\Gamma_0\times\R^d$,
	\item[$\bullet$]$(\hat{b}_{n})_{n\geq 1}$ (respctively, $(\check{b}_{n})_{n\geq 1}$) converges to  $\hat{b}$ (respctively, $\check{b}$) in $(s,t,x) \in \Gamma_0\times \R^d$ $\mathrm{d}s\times\mathrm{d}t\times\mathrm{d}x$-a.e.
\end{enumerate}
One verifies that for such smooth drift coefficients, the corresponding SDEs have a unique strong solution denoted by $X^{\xi,n}$. We show that for $s,t$ small enough, the sequence $(X^{\xi,n}_{s,t})_{n \geq 1}$ is relatively compact in $L^2(\Omega,\Pb;\R^d)$. The proofs of the next two results are similar to that of Lemmas \ref{lemMalEstimate1} and \ref{thmcompact2} and are found in Appendix.
\begin{lemma}\label{lemMalDifflgrowth}
	There exist $\widetilde{C}_2>0$ and $\ze>0$ such that, for any $i,j\in\{1,\cdots,d\}$, any $0< r<s\leq1$, any $0< u<t\leq1$ and any $k\in\R_+$,  
	\begin{align}\label{eqDavie12}
		\mathbb{E}\Big[\exp\Big(\frac{\ze}{\delta(r,s)\delta(u,t)} \int_r^s\int_u^t\partial_i\hat{b}_{j,n}(s_1,t_1,W_{s_1,t_1})\mathrm{d}t_1\mathrm{d}s_1\Big)\Big]\leq \widetilde{C}_2,
	\end{align}
	and
	\begin{align}\label{eqDavie22}
		\mathbb{E}\Big[\exp\Big(\frac{\ze}{\delta(r,s)\delta(u,t)} \int_r^s\int_u^t\partial_i\check{b}_{j,n}(s_1,t_1,W_{s_1,t_1})\mathrm{d}t_1\mathrm{d}s_1\Big)\Big]\leq \widetilde{C}_2,
	\end{align}
	where $\delta(r,s)=\sqrt{s-r}$ and $\delta(u,t)=\sqrt{t-u}$.
\end{lemma}

\begin{lemma}\label{thmcompact22}
	There exist $C_2>0$ and $\tau\in(0,1)$ such that for every $(s,t)\in[0,\tau]$, the sequence $(X^{\xi,n}_{s,t} )_{n\geq1}$   satisfies
	\begin{align}\label{eqcompact32}
		\sup_{n\geq1} \Vert  X^n_{s,t}\Vert_{L^2(\Omega,\Pb;\R^d)}^2 \leq  C_2
	\end{align}
	and
	\begin{align}\label{eqcompact22}
		\sup_{n\geq1}	\sup_{
			\begin{subarray}{c}
				0 \leq r \leq s\\
				0 \leq u \leq t
		\end{subarray}} \E \left[ \Vert D_{r,u}X^n_{s,t}\Vert^2 \right] \leq  C_2.
	\end{align}
	Moreover, for all $0 \leq r',r \leq s \leq \tau,\,\,0 \leq u',u \leq t \leq \tau$,
	\begin{align}\label{eqcompact12}
		\E \left[ \Vert D_{r,u}X^n_{s,t} - D_{r',u'}X^n_{s,t} \Vert^2 \right] \leq C_2(|r -r'|^{} +|u -u'|^{}).
	\end{align}
\end{lemma}
Here is the main result of this section which is a consequence of Lemma \ref{thmcompact22} and the compactness criterion provided in Corollary \ref{CorolCompact}.
\begin{theorem}\label{themmalldiff12}
	There exist $\tau\in(0,1)$ such that for any  the strong solution $\{X^{\xi}_{s,t},\,s,t\in[0,\tau]\}$ of the SDE \eqref{Eqmainhpde5} is Malliavin differentiable. 
\end{theorem}
\begin{proof}
	As in the proof of Lemma \ref{thml2con2}, we show that $(X_{s,t}^{\xi,n})_{n\geq1}$ converges weackly in $L^2(\Omega,\Pb;\R^d)$ to $X^{\xi}_{s,t}$ for every $(s,t)\in[0,\tau]^2$. Hence, using Lemma \ref{thmcompact22} and the compactness criterion provided in Corollary \ref{CorolCompact}, we deduce that $(X_{s,t}^{\xi,n})_{n\geq1}$ converges strongly in $L^2(\Omega,\Pb;\R^d)$ to $X^{\xi}_{s,t}$. Then, since $(D_{r,u}X^{\xi,n}_{s,t})_{n\geq 1}$ is bounded in the $L^2([0,1]^2\times \Omega,\mathrm{d}s\times\mathrm{dt}\times\Pb;\R^{d\times d})$-norm uniformly in $n$ (see \eqref{eqcompact22}), it follows from \cite[Lemma 1.2.3]{Nu06} that the limit $X^{\xi}_{s,t}$ is also 
	Malliavin differentiable.
\end{proof}

\begin{appendix}	
	\section*{Appendix}\label{appn} 
	In this section we provide the proofs of Lemmas \ref{lemMalDifflgrowth} and \ref{thmcompact22}. Let us start with a useful estimate.
	\begin{lemma}\label{lemUnifL2ExpBound}
		There exist $\tau_1\in(0,1)$ such that 
		\begin{align}\label{UnifL2ExpBound}
			\sup\limits_{n\geq1}\E\Big[\mathcal{E}\Big(\int_0^{\tau_1}\int_{0}^{\tau_1}b_n(s,t,W_{s,t})\cdot\mathrm{d}W_{s,t}\Big)^2\Big]<\infty.
		\end{align}
	\end{lemma}
	\begin{proof}
		By Cauchy-Schwarz inequality, we have
		\begin{align*}
			&\E\Big[\mathcal{E}\Big(\int_0^{\tau_1}\int_{0}^{\tau_1}b_n(s,t,W_{s,t})\cdot\mathrm{d}W_{s,t}\Big)^2\Big]\\=&\E\Big[\exp\Big(2\int_0^{\tau_1}\int_{0}^{\tau_1}b_n(s,t,W_{s,t})\cdot\mathrm{d}W_{s,t}- \int_{0}^{\tau_1}\int_0^{\tau_1}|b_n(s,t,W_{s,t})|^2\mathrm{d}s\mathrm{d}t\Big)\Big]\\=&\E\Big[\exp\Big(2\int_0^{\tau_1}\int_{0}^{\tau_1}b_n(s,t,W_{s,t})\cdot\mathrm{d}W_{s,t}- 4\int_{0}^{\tau_1}\int_0^{\tau_1}|b_n(s,t,W_{s,t})|^2\mathrm{d}s\mathrm{d}t+3\int_{0}^{\tau_1}\int_0^{\tau_1}|b_n(s,t,W_{s,t})|^2\mathrm{d}s\mathrm{d}t\Big)\Big]\\ \leq&\E\Big[\mathcal{E}\Big(4\int_0^{\tau_1}\int_{0}^{\tau_1}b_n(s,t,W_{s,t})\cdot\mathrm{d}W_{s,t} \Big)\Big]^{\frac{1}{2}}\E\Big[\exp\Big(6\int_{0}^{\tau_1}\int_0^{\tau_1}|b_n(s,t,W_{s,t})|^2\mathrm{d}s\mathrm{d}t\Big)\Big]^{\frac{1}{2}}.
		\end{align*}
		Since $b_n$ is bounded, we have (see e.g. \cite[Proposition 1.6]{NP94})
		\begin{align*}
			\E\Big[\mathcal{E}\Big(4\int_0^{\tau_1}\int_{0}^{\tau_1}b_n(s,t,W_{s,t})\cdot\mathrm{d}W_{s,t} \Big)\Big]=1,\quad\forall\,n\geq1,\,\tau_1>0.
		\end{align*}
		Moreover, by Jensen inequality, 
		\begin{align*}
			\E\Big[\exp\Big(6\int_{0}^{\tau_1}\int_0^{\tau_1}|b_n(s,t,W_{s,t})|^2\mathrm{d}s\mathrm{d}t\Big)\Big]\leq&\frac{1}{\tau_1^2}\int_{0}^{\tau_1}\int_0^{\tau_1}\E\Big[\exp\left(6\tau_1^2|b_n(s,t,W_{s,t})|^2\right)\Big]\mathrm{d}s\mathrm{d}t\\ \leq&\frac{1}{\tau_1^2}\int_{0}^{\tau_1}\int_0^{\tau_1}\E\Big[\exp\left(24\tau_1^2\tilde{M}d(1+|W_{s,t}|^2)\right)\Big]\mathrm{d}s\mathrm{d}t\\=&\frac{\exp\left(24\tau_1^2\tilde{M}d\right)}{\tau_1^2}\int_{0}^{\tau_1}\int_0^{\tau_1}\E\left[\exp(24\tau_1^2\tilde{M}d\,st|G|^2)\right]\mathrm{d}s\mathrm{d}t\\ \leq& \exp\left(24\tau_1^2\tilde{M}d\right) \E\left[\exp(24\tau_1^4\tilde{M}d|G|^2)\right],
		\end{align*}
		where $G=(G_1,\ldots,G_d)$ is a reduced Gaussian random vector. 
		
		The proof is completed since $\E\left[\exp(12\tau_1^4\tilde{M}d|G|^2)\right]<\infty$ for $\tau_1$ small enough.
	\end{proof}
	
	\begin{remark}\label{remUnifL2ExpBound}
		When the drift $b$ is bounded, the functions $b_n$, $n\geq1$ are uniformly bounded and, as a consequence,
		\begin{align*} 
			C_0=\sup\limits_{n\geq1}\E\Big[\mathcal{E}\Big(\int_0^{1}\int_{0}^{1}b_n(s,t,W_{s,t})\cdot\mathrm{d}W_{s,t}\Big)^2\Big]^\frac{1}{2}<\infty.
		\end{align*} 
	\end{remark}
	
	\begin{proof}[Proof of Lemma \ref{lemMalDifflgrowth}.]
		We only prove \eqref{eqDavie12} since the proof of \eqref{eqDavie22} follows the same lines. We deduce from the local time-space integration formula \eqref{eq:EisenSheetdD01} that  
		\begin{align}\label{IneqDWest2} 
			&\mathbb{E}\Big[\exp\Big(\frac{\ze}{\delta(r,s)\delta(u,t)}\Big|\int_r^s\int_u^t\partial_i\hat{b}_{j,n}(s_1,t_1,W_{s_1,t_1})\mathrm{d}t_1\mathrm{d}s_1\Big|\Big)\Big]\nonumber\\
			\leq&\frac{1}{3}\Big\{\mathbb{E}\Big[\exp\Big(\frac{3\ze}{\delta(r,s)\delta(u,t)}\Big|\int_r^{s}\int_u^{t}\hat{b}_{j,n}(s_1,t_1,W_{s_1,t_1})\frac{\mathrm{d}_{t_1}W^{(i)}_{s_1,t_1}}{s_1}\mathrm{d}s_1\Big|\Big)\Big]\\&\quad+ \mathbb{E}\Big[\exp\Big(\frac{3\ze}{\delta(r,s)\delta(u,t)}\Big|\int_r^{s}\int_{1-t}^{1-u}\hat{b}_{j,n}(s_1,t_1,W_{s_1,1-t_1})\frac{\mathrm{d}_{t_1}B^{(i)}_{s_1,t_1}}{s_1}\mathrm{d}s_1 \Big|\Big)\Big]\nonumber\\
			&\quad+\mathbb{E}\Big[\exp\Big(\frac{3\ze}{\delta(r,s)\delta(u,t)}\Big|\int_r^{s}\int_{1-t}^{1-u}\hat{b}_{j,n}(s_1,t_1,W_{s_1,1-t_1})\frac{W^{(i)}_{s_1,1-t_1}}{s_1(1-t_1)}\mathrm{d}t_1\mathrm{d}s_1\Big|\Big)\Big]\Big\}=\frac{1}{3}(I_1+I_2+I_3).\nonumber
		\end{align}
		By Jensen inequality,
		\begin{align*}
			I_1=&\mathbb{E}\Big[\exp\Big(\frac{3\ze}{\delta(r,s)\delta(u,t)}\Big|\int_r^{s}\int_u^{t}\hat{b}_{j,n}(s_1,t_1,W_{s_1,t_1})\frac{\mathrm{d}_{t_1}W^{(i)}_{s_1,t_1}}{s_1}\mathrm{d}s_1\Big|\Big)\Big]\\\leq&\int_r^s\mathbb{E}\Big[\exp\Big(\frac{6\ze(\sqrt{s}-\sqrt{r})}{\delta(r,s)\delta(u,t)}\Big|\int_u^{t}\hat{b}_{j,n}(s_1,t_1,W_{s_1,t_1})\frac{\mathrm{d}_{t_1}W^{(i)}_{s_1,t_1}}{\sqrt{s_1}}\Big|\Big)\Big] \frac{\mathrm{d}s_1}{2\sqrt{s_1}(\sqrt{s}-\sqrt{r})}\\
			\leq&\int_r^s\mathbb{E}\Big[\exp\Big(\frac{6\ze}{\delta(u,t)}\Big|\int_u^{t}\hat{b}_{j,n}(s_1,t_1,W_{s_1,t_1})\frac{\mathrm{d}_{t_1}W^{(i)}_{s_1,t_1}}{\sqrt{s_1}}\Big|\Big)\Big] \frac{\mathrm{d}s_1}{2\sqrt{s_1}(\sqrt{s}-\sqrt{r})}
		\end{align*}	
		Since, for every $s_1\in[r,s]$, $$\Big(Y_{s_1,t_1}:=\int_{u}^{t_1}\hat{b}_{j,n}(s_1,t_2,W_{s_1,t_2})\frac{\mathrm{d}_{t_2}W^{(i)}_{s_1,t_2}}{\sqrt{s_1}},u\leq t_1\leq t\Big)$$ is a square integrable martingale and similar reasoning as before gives
		%	it follows from the Barlow-Yor inequality and Jensen inequality that there exists a positive constant $c_1$ such that, for any positive integer $m$ and any $s_1\in[r,s]$,
		\begin{align*}
			\E[|Y_{s_1,t}|^m] &\leq 2^mc_1^{2m}\tilde{M}^mm^{m}\delta(u,t)^{m}.
		\end{align*}
		From this and the exponential expansion formula, we get 
		%\begin{align*}
		%	\exp\Big(\frac{6\ze}{\delta(u,t)}|Y_{s_1,t}|\Big)=\sum\limits_{m=0}^{\infty}\frac{6^m\ze^m|Y_{s_1,t}|^m}{\delta(u,t)^m\,m!},
		%	\end{align*}
		%	we obtain
		\begin{align*}
			I_1\leq&%\int_r^s\mathbb{E}\Big[\exp\Big(\frac{6\ze}{\delta(u,t)}\Big|\int_u^{t}\hat{b}_{j,n}(s_1,t_1,W_{s_1,t_1})\frac{\mathrm{d}_{t_1}W_{s_1,t_1}}{\sqrt{s_1}}\Big|\Big)\Big] \frac{\mathrm{d}s_1}{2\sqrt{s_1}(\sqrt{s}-\sqrt{r})}\\=&\int_r^s\E\Big[\exp\Big(\frac{6\ze}{\delta(u,t)}|Y_{s_1,t}|\Big)\Big]\frac{\mathrm{d}s_1}{2\sqrt{s_1}(\sqrt{s}-\sqrt{r})}=\int_r^s\sum\limits_{m=0}^{\infty}\frac{6^m\ze^m\E[|Y_{s_1,t}|^m]}{\delta(u,t)^mm!}\,\frac{\mathrm{d}s_1}{2\sqrt{s_1}(\sqrt{s}-\sqrt{r})}\\ \leq&
			\sum\limits_{m=0}^{\infty}\frac{2^{4m}\ze^mc_1^{2m}m^{m}\tilde{M}^m}{m!}:=\widetilde{C}_{1,1},
		\end{align*}
		which is finite if $\ze<e/2^4c_1^2\tilde{M}$ (by ratio test).   Similarly, we also have
		\begin{align*}
			I_2\leq\widetilde{C}_{1,1}.
		\end{align*}
		To estimate $I_3$ we apply Jensen inequality again and we obtain
		\begin{align*}
			I_3=& \mathbb{E}\Big[\exp\Big(\frac{3\ze}{\delta(r,s)\delta(u,t)}\Big|\int_r^{s}\int_{1-t}^{1-u}\hat{b}_{j,n}(s_1,t_1,W_{s_1,1-t_1})\frac{W^{(i)}_{s_1,1-t_1}}{s_1(1-t_1)}\mathrm{d}t_1\mathrm{d}s_1\Big|\Big)\Big]\\
			\leq&\int_r^{s}\int_{1-t}^{1-u}\mathbb{E}\Big[\exp\Big(\frac{12\ze(\sqrt{s}-\sqrt{r})(\sqrt{1-u}-\sqrt{1-t})}{\delta(r,s)\delta(u,t)}|\hat{b}_{j,n}(s_1,t_1,W_{s_1,1-t_1})|\Big|\frac{W^{(i)}_{s_1,1-t_1}}{\sqrt{s_1}\sqrt{1-t_1}}\Big|\Big)\Big]\\&\qquad\qquad\times\frac{\mathrm{d}t_1}{2\sqrt{1-t_1}(\sqrt{1-u}-\sqrt{1-t})}\frac{\mathrm{d}s_1}{2\sqrt{s_1}(\sqrt{s}-\sqrt{r})}\\
			\leq&\int_r^{s}\int_{1-t}^{1-u}\mathbb{E}\Big[\exp\Big\{24\tilde{M}\ze\Big(1+ \frac{|W^{(i)}_{s_1,1-t_1}|^2}{s_1(1-t_1)}\Big)\Big\}\Big]\frac{\mathrm{d}t_1}{2\sqrt{1-t_1}(\sqrt{1-u}-\sqrt{1-t})}\frac{\mathrm{d}s_1}{2\sqrt{s_1}(\sqrt{s}-\sqrt{r})}
			:=\widetilde{C}_{1,2},
		\end{align*}
		%Since $\Big|\frac{W^{(i)}_{s_1,1-t_1}}{\sqrt{s_1}\sqrt{1-t_1}}\Big|$ is a reduced Gaussian random variable, 
		with $\widetilde{C}_{1,2}$ finite provided that $24\tilde{M}\ze<1/2$.\\ The proof is completed by choosing  $\ze=1/2^6c_1^2\tilde{M}$ and $\widetilde{C}_1=(2\widetilde{C}_{1,1}+\widetilde{C}_{1,2})/3$.
	\end{proof}
	\begin{proof}[Proof of Lemma \ref{thmcompact22}.]
		Fix $\tau\in(0,\min\{\tau_1,\ze/32d^2\})$, where $\tau_1$ is the constant in Lemma \ref{lemUnifL2ExpBound} and $\ze$ is the constant in Lemma \ref{lemMalDifflgrowth}. 
		We deduce from \cite[Lemma 5.1.1]{Qi16} and the linear growth condition on the drift $b_n$ that $\E[\|X^n_{s,t}\|^2]\leq C_{2,1}$ for all $(s,t)\in[0,\tau]^2$ and $n\geq1$, where $C_{2,1}$ does not depend of $(s,t)$ and $n$. Let $0 \leq r'\leq r \leq s \leq \tau$ and $0 \leq u'\leq u \leq t \leq \tau$. Since $$\tau\leq\tau_1\,\text{ and }\,\delta(r',s)\delta(u',t)\leq\tau\leq \frac{\ze}{32d^2},$$
		then, using similar computations as in the proof of Lemma \ref{thmcompact2}, one can deduce from Lemma \ref{lemUnifL2ExpBound}, Girsanov theorem and H\"older inequality that
		\begin{align*}
			\sup\limits_{n\geq1}\,\sup\limits_{
				\begin{subarray}{c}
					0\leq r\leq s\\
					0\leq u\leq t
			\end{subarray}}\E\Big[\|D_{r,u}X^n_{s,t}\|^2\Big]:=C_{2,2}<\infty
		\end{align*}
		and
		\begin{align*}
			\E\Big[\|D_{r,u}X^n_{s,t}-D_{r',u'}X^n_{s,t}\|^2\Big]\leq C_{2,3}(|r-r'|+|u-u'|)
		\end{align*}
		for some positive constant $\widetilde{C}_2$ independent of $n$. The proof is completed by taking $C_2=\max\{C_{2,1},C_{2,2},C_{2,3}\}$.
	\end{proof}
	
	\begin{remark}
		It is worth noting that if the drift is in addition the difference of two convex or concave functions, then the  solution to the equation \eqref{Eqmainhpde5} is twice Malliavin differentiable. Note that is this case,drfit $b=\hat{b}-\check{b}$ is Lipschitz with the second order weak derivatives of $\hat{b},\check{b}$  positive or negative. 
	\end{remark} 
	%% if no title is needed, leave empty \section*{}.
	%Appendices should be provided in \verb|{appendix}| environment,
	%before Acknowledgements.
	%
	%If there is only one appendix,
	%then please refer to it in text as \ldots\ in the \hyperref[appn]{Appendix}.
\end{appendix}

\end{document}